\documentclass[11pt]{amsart}

\newtheorem{thm}{Theorem}[section]
\newtheorem*{thm*}{Theorem}

\newtheorem*{lem*}{Lemma}
\newtheorem{mainthm}{Theorem}
\newtheorem*{mainthm*}{Theorem}
\newtheorem{maincor}[mainthm]{Corollary}

\newtheorem{prop}[thm]{Proposition}

\theoremstyle{definition}

\newtheorem*{case*}{Case}

\newtheorem{defn}[thm]{Definition}
\newtheorem*{defn*}{Definition}

\newtheorem*{exmp*}{Example}

\renewcommand{\thestep}{}

\theoremstyle{remark}

\renewcommand{\thecase}{}

\newtheorem{rmk}[thm]{Remark}
\newtheorem*{rmk*}{Remark}

\makeatletter
\def\alphenumi{
  \def\theenumi{\alph{enumi}}
  \def\p@enumi{\theenumi}
  \def\labelenumi{(\@alph\c@enumi)}}
\makeatother

\makeatletter
\def\thecase{\@arabic\c@case}
\makeatother

\makeatletter
\def\thestep{\@arabic\c@step}
\makeatother

\newcount\hh
\newcount\mm
\mm=\time
\hh=\time
\divide\hh by 60
\divide\mm by 60
\multiply\mm by 60
\mm=-\mm
\advance\mm by \time
\def\hhmm{\number\hh:\ifnum\mm<10{}0\fi\number\mm}

\let\oldmarginpar\marginpar
\renewcommand\marginpar[1]{\-\oldmarginpar[\raggedleft\footnotesize #1]%
{\raggedright\footnotesize #1}}

\newcommand\CC{\mathbb{C}}

\newcommand\KK{\mathbb{K}}

\newcommand\PP{\mathbb{P}}

\newcommand\RR{\mathbb{R}}

\newcommand\ZZ{\mathbb{Z}}

\newcommand\fm{{\mathfrak{m}}}

\newcommand\sE{{\mathscr{E}}}
\newcommand\sF{{\mathscr{F}}}

\newcommand\sI{{\mathscr{I}}}

\newcommand\sO{{\mathscr{O}}}

\newcommand\bx{{\mathbf{x}}}

\newcommand\by{{\mathbf{y}}}

\newcommand\less{\setminus}

\newcommand\cosupp{\operatorname{cosupp}}

\DeclareMathOperator{\Crit}{Crit}

\newcommand\dist{\operatorname{dist}}

\newcommand\End{\operatorname{End}}

\newcommand\Hess{\operatorname{Hess}}

\newcommand\Hom{\operatorname{Hom}}

\newcommand\Ker{\operatorname{Ker}}
\newcommand\Length{\operatorname{Length}}

\newcommand\supp{\operatorname{supp}}

\DeclareMathOperator{\Zero}{Zero}

\numberwithin{equation}{section}
\numberwithin{figure}{section}

\usepackage{graphicx}
\usepackage{hyperref}
\usepackage{mathrsfs}
\usepackage[usenames]{color}
\usepackage{paralist}
\usepackage{slashed}
\usepackage{txfonts}
\usepackage{url}
\usepackage{verbatim}
\usepackage{xr}

\usepackage{amsmath}
\usepackage{amscd}
\usepackage{amssymb}
\usepackage{lmodern}

\makeatletter
\newcommand{\tpitchfork}{%
  \vbox{
    \baselineskip\z@skip
    \lineskip-.52ex
    \lineskiplimit\maxdimen
    \m@th
    \ialign{##\crcr\hidewidth\smash{$-$}\hidewidth\crcr$\pitchfork$\crcr}
  }%
}
\makeatother

\hypersetup{pdftitle={Resolution of Singularities and Geometric Proofs of the Lojasiewicz Inequalities}}
\hypersetup{pdfauthor={Paul M. N. Feehan}}

\begin{document}

\title[Geometric Proofs of the {\L}ojasiewicz Inequalities]{Resolution of Singularities and Geometric Proofs of the {\L}ojasiewicz Inequalities}

\author[Paul M. N. Feehan]{Paul M. N. Feehan}
\address{Department of Mathematics, Rutgers, The State University of New Jersey, 110 Frelinghuysen Road, Piscataway, NJ 08854-8019, United States of America}
\email{feehan@math.rutgers.edu}

\date{This version: January 6, 2020, incorporating final galley proof corrections. \emph{Geometry \& Topology} \textbf{23} (2019), no. 7, 3273--3313, \url{https://doi.org/10.2140/gt.2019.23.3273} and \url{https://arxiv.org/abs/1708.09775}.}

\begin{abstract}
The {\L}ojasiewicz inequalities for real analytic functions on Euclidean space were first proved by Stanis{\l}aw {\L}ojasiewicz in \cite{Lojasiewicz_1959, Lojasiewicz_1961, Lojasiewicz_1965} using methods of semianalytic and subanalytic sets, arguments later simplified by Bierstone and Milman \cite{BierstoneMilman}. Here, we first give an elementary geometric, coordinate-based proof of the {\L}ojasiewicz inequalities in the special case where the function is $C^1$ with simple normal crossings. We then prove, partly following Bierstone and Milman \cite[Section 2]{Bierstone_Milman_1997} and using resolution of singularities for (real or complex) analytic varieties, that the gradient inequality for an arbitrary analytic function follows from the special case where it has simple normal crossings. In addition, we prove the {\L}ojasiewicz inequalities when a function is $C^N$ and generalized Morse--Bott of order $N\geq 3$; we gave an elementary proof of the {\L}ojasiewicz inequalities when a function is $C^2$ and Morse--Bott on a Banach space in \cite{Feehan_lojasiewicz_inequality_ground_state}.
\end{abstract}

\subjclass[2010]{Primary 32B20, 32C05, 32C18, 32C25, 58E05; secondary 14E15, 32S45, 57R45, 58A07, 58A35}

\keywords{Analytic varieties, {\L}ojasiewicz inequalities, gradient flow, Morse--Bott functions, resolution of singularities, semianalytic sets and subanalytic sets}

\thanks{The author was partially supported by National Science Foundation grant DMS-1510064 and the Simons Center for Geometry and Physics, Stony Brook, the Dublin Institute for Advanced Studies, the Institut des Hautes {\'E}tudes Scientifiques, Bures-sur-Yvette, and the Institute for Advanced Studies, Princeton.}

\maketitle


\section{Introduction}
\label{sec:Introduction}
Our goal is to provide geometric proofs of the {\L}ojasiewicz inequalities (Theorem \ref{mainthm:Lojasiewicz_gradient_inequality} and Corollaries \ref{maincor:Lojasiewicz_distance_inequality} and \ref{maincor:Lojasiewicz_gradient-distance_inequality}) for functions with simple normal crossings and hence, via resolution of singularities, for arbitrary analytic functions on (real or complex) Euclidean space. In contrast, for a function that is (generalized) Morse--Bott (so its critical set is a submanifold), elementary methods suffice to prove the {\L}ojasiewicz inequalities (Theorems \ref{thm:Lojasiewicz_gradient_inequality_Morse-Bott} and \ref{thm:Lojasiewicz_gradient_inequality_generalized_Morse-Bott_function}).

The original proofs by Stanis{\l}aw {\L}ojasiewicz of his inequalities \cite{Lojasiewicz_1963, Lojasiewicz_1964, Lojasiewicz_1965, Lojasiewicz_1970, Lojasiewicz_1993} relied on the theory of semianalytic sets and subanalytic sets originated by him and further developed by Gabri\`elov \cite{Gabrielov_1968}, Hardt \cite{Hardt_1975, Hardt_1977} and Hironaka \cite{Hironaka_1973, Hironaka_1977, Hironaka_intro_real-analytic_sets_maps}. The proofs due to {\L}ojasiewicz of his inequalities are well-known to be technically difficult. The most accessible modern approaches to the inequalities were provided by Bierstone and Milman. In \cite{BierstoneMilman}, they significantly simplify the {\L}ojasiewicz theory of semianalytic sets and subanalytic sets and prove his gradient inequality as a consequence of technical results in that theory. In \cite{Bierstone_Milman_1997}, they develop an approach to resolution of singularities for algebraic and analytic varieties over a field of characteristic zero that relies on blowing up and greatly simplifies the original arguments due to Hironaka et al. \cite{Aroca_Hironaka_Vicente_theory_maximal_contact, Aroca_Hironaka_Vicente_desingularization_theorems, Hironaka_1964-I-II, Hironaka_infinitely_near_singular_points}. They then deduce the {\L}ojasiewicz gradient inequality as a consequence of resolution of singularities for analytic varieties and a direct verification when the critical and zero set of an analytic function is a simple normal crossing divisor.

The {\L}ojasiewicz gradient inequality was generalized by Leon Simon \cite{Simon_1983} to a certain class of real analytic functions on a H\"older space of $C^{2,\alpha}$ sections of a finite-rank vector bundle over a closed, finite-dimensional smooth manifold. Simon's proof relied on a splitting (or Lyapunov--Schmidt reduction) of the real analytic function into a finite-dimensional part, to which the original {\L}ojasiewicz gradient inequality could be applied, and a benign infinite-dimensional part. The resulting {\L}ojasiewicz-Simon gradient inequality and its many generalizations and variants have played a significant role in analyzing questions such as
\begin{inparaenum}[\itshape a\upshape)]
\item global existence, convergence, and analysis of singularities for solutions to nonlinear evolution equations that are realizable as gradient-like systems for an energy function,
\item uniqueness of tangent cones, and
\item gap theorems.
\end{inparaenum}
See Feehan \cite{Feehan_yang_mills_gradient_flow_v4}, Feehan and Maridakis \cite{Feehan_Maridakis_Lojasiewicz-Simon_Banach, Feehan_Maridakis_Lojasiewicz-Simon_coupled_Yang-Mills_v4}, and Huang \cite{Huang_2006} for references and a survey of {\L}ojasiewicz--Simon gradient inequalities for analytic functions on Banach spaces and their many applications in applied mathematics, geometric analysis, and mathematical physics.

Our hope is that the more geometric and direct coordinate-based approaches provided in this article to proofs of the {\L}ojasiewicz gradient inequality may yield greater insight that could be useful when endeavoring to prove gradient inequalities for functions on Banach spaces arising in geometric analysis without relying on Lyapunov--Schmidt reduction to the gradient inequality for functions on Euclidean space or attempting to extend methods specific to algebraic geometry. For example, the {\L}ojasiewicz inequalities for the $F$ functional on the space of hypersurfaces in Euclidean space are proved directly by Colding and Minicozzi \cite{Colding_Minicozzi_2014sdg, Colding_Minicozzi_2015, Colding_Minicozzi_Pedersen_2015bams} and by the author for the Yang--Mills energy function near regular points in the moduli space of flat connections on a principal $G$-bundle over a closed, smooth Riemannian manifold \cite{Feehan_lojasiewicz_inequality_ground_state}. Applications in geometric analysis typically concern functions on infinite-dimensional manifolds and, in that context, arguments specific to semianalytic sets or subanalytic subsets or real analytic subvarieties of Euclidean space do not necessarily have analogues in infinite-dimensional geometry. Like Bierstone and Milman in \cite[Section 2]{Bierstone_Milman_1997}, we ultimately apply resolution of singularities to obtain the {\L}ojasiewicz gradient inequality for an arbitrary analytic function, but after directly proving the gradient inequality in simpler cases. When the function is $C^N$ and Morse--Bott of order $N\geq 2$, we obtain a {\L}ojasiewicz exponent $\theta=1-1/N$ (see Theorems \ref{thm:Lojasiewicz_gradient_inequality_Morse-Bott} and \ref{thm:Lojasiewicz_gradient_inequality_generalized_Morse-Bott_function}) and when the function is $C^1$ with simple normal crossings, we obtain an explicit bound for the {\L}ojasiewicz exponent --- which implies that $\theta \in [1/2,1)$ rather than $\theta \in (0,1)$ --- together with a characterization of when $\theta$ has the optimal value $1/2$.

We showed in \cite[Section 4]{Feehan_lojasiewicz_inequality_ground_state} that one can use the Mean Value Theorem to prove the {\L}ojasiewicz gradient inequality for a $C^2$ Morse--Bott function on a Banach space in a context of wide applicability
\cite[Theorem 3]{Feehan_lojasiewicz_inequality_ground_state}. The facts that a Morse--Bott function has a critical set which is a smooth submanifold and a Hessian which is non-degenerate on the normal bundle ensure that the Mean Value Theorem easily yields the {\L}ojasiewicz gradient inequality (with optimal {\L}ojasiewicz exponent $1/2$). In Section \ref{sec:Lojasiewicz_gradient_inequality_critical_normal_crossing_divisor}, we prove that the {\L}ojasiewicz gradient inequality (Theorem \ref{mainthm:Lojasiewicz_gradient_inequality_normal_crossing_divisor}) holds for a $C^1$ function that has simple normal crossings in the sense of Definition \ref{defn:C1_function_normal_crossings}. We then appeal to resolution of singularities (Theorem \ref{thm:Monomialization_ideal_sheaf}) to show that the {\L}ojasiewicz gradient inequality for an arbitrary analytic function, Theorem \ref{mainthm:Lojasiewicz_gradient_inequality}, is a straightforward consequence of Theorem \ref{mainthm:Lojasiewicz_gradient_inequality_normal_crossing_divisor}. This incremental approach makes it clear that the essential difficulty is due neither to the high dimension of the ambient Euclidean space nor the critical set, but instead due (as should be expected) to possibly complicated singularities in the critical set.

Simplifications of {\L}ojasiewicz's proofs \cite{Lojasiewicz_1965} of his inequalities have also been given by Kurdyka and Parusi{\'n}ski \cite{Kurdyka_Parusinski_1994}, where they use the fact that a subanalytic set in Euclidean space admits a strict Thom stratification. {\L}ojasiewicz and Zurro \cite{Lojasiewicz_Zurro_1999} further simplified the arguments of Kurdyka and Parusi{\'n}ski to prove the {\L}ojasiewicz inequalities, again using properties of subanalytic sets.

The problem of estimating {\L}ojasiewicz exponents or determining their properties, often for restricted classes of functions (for example, polynomials, certain analytic functions, functions with isolated critical points, and so on), has been pursued by many researchers, including Abderrahmane \cite{Abderrahmane_2005}, Bivi\`a-Ausina \cite{Bivia-Ausina_2015}, Bivi\`a-Ausina and Encinas \cite{Bivia-Ausina_Encinas_2009, Bivia-Ausina_Encinas_2011, Bivia-Ausina_Encinas_2013}, Bivi\`a-Ausina and Fukui \cite{Bivia-Ausina_Fukui_2016}, Brzostowski \cite{Brzostowski_2015}, Brzostowski, Krasi\'nski, and Oleksik \cite{Brzostowski_Krasinski_Oleksik_2012}, B\'ui and Pham \cite {Bui_Pham_2014}, D'Acunto and Kurdyka \cite{dAcunto_Kurdyka_2005}, Fukui \cite{Fukui_1991}, Gabri\`elov \cite{Gabrielov_1995} Gwo\'zdziewicz \cite{Gwozdziewicz_1999}, Haraux \cite[Theorem 3.1]{Haraux_2005}, Haraux and Pham \cite{Haraux_Pham_2011, Haraux_Pham_2015}, Ji, Koll{\'a}r, and Shiffman \cite{Ji_Kollar_Shiffman_1992} Koll\'ar \cite{Kollar_1999}, Krasi\'nski, Oleksik, and P{\l}oski \cite{Krasinski_Oleksik_Ploski_2009}, Kuo \cite{Kuo_HH_1974}, Lichtin \cite{Lichtin_1981}, Lenarcik \cite{Lenarcik_1998, Lenarcik_1999}, Lion \cite{Lion_2000}, Oka \cite{Oka_2017arxiv}, Oleksik \cite{Oleksik_2009}, Pham \cite{Pham_2011, Pham_2012}, P{\l}oski \cite{Ploski_2009}, Risler and Trotman \cite{Risler_Trotman_1997}, Rodak and Spodzieja \cite{Rodak_Spodzieja_2011}, \cite{Spodzieja_2005}, Tan, Yau, and Zuo \cite{Tan_Yau_Zuo_2010}, and Teissier \cite{Teissier_1977}. Recently, simpler coordinate-based proofs of more limited versions of resolution of singularities for zero sets of real analytic functions, with applications to analysis, have been given by Collins, Greenleaf, and Pramanik \cite{Collins_Greenleaf_Pramanik_2013} and Greenblatt \cite{Greenblatt_2008}. In particular, Greenblatt  \cite[p. 1959]{Greenblatt_2008} applies his version of resolution of singularities to prove the {\L}ojasiewicz inequality \eqref{eq:Lojasiewicz_division_inequality_pair_functions} for a pair of real analytic functions where the zero set of one is contained in the zero set of the other. Bivi\`a-Ausina and Encinas \cite{Bivia-Ausina_Encinas_2009} use a resolution of singularities algorithm to estimate {\L}ojasiewicz exponents.

{\L}ojasiewicz \cite{Lojasiewicz_1959, Lojasiewicz_1961} applied his distance inequality (Corollary \ref{maincor:Lojasiewicz_distance_inequality}) to prove the Division Conjecture of Schwartz \cite[p. 181]{Schwartz_1955},
\cite[p. 116]{Schwartz_distributions_v1}. In \cite{Lojasiewicz_1963}, he used his gradient inequality (Theorem \ref{mainthm:Lojasiewicz_gradient_inequality}) to give a positive answer \cite[Theorem 5]{Lojasiewicz_1963} to a question of Whitney: If $\sE$ is a real analytic function on an open set $U \subset \RR^d$, then $\sE^{-1}(0)$ is a deformation retract of its neighborhood. This deformation retract is obtained using the negative gradient flow defined by $\sE$. He also applies his inequalities to show that every (locally closed) semianalytic set in Euclidean space admits a Whitney stratification\footnote{The first page number refers to the version of {\L}ojasiewicz's original manuscript mimeographed by IHES while the page number in parentheses refers to the cited LaTeX version of his manuscript prepared by M. Coste and available on the Internet.} \cite[Proposition 3, p. 97 (71)]{Lojasiewicz_1965}. The {\L}ojasiewicz gradient inequality (Theorem \ref{mainthm:Lojasiewicz_gradient_inequality}) was used by Kurdyka, Mostowski, and Parusi{\'n}ski \cite{Kurdyka_Mostowski_Parusinski_2000} to prove the Gradient Conjecture of Thom.

Atiyah \cite{Atiyah_1970} and Bernstein and Gelfand \cite{Bernstein_Gelfand_1969} appear to be the first authors to have noticed that resolution of singularities could be used to simplify proofs of {\L}ojasiewicz's results, a fact that we discovered only when correcting galley proofs for this article. In \cite{Atiyah_1970}, Atiyah employed resolution of singularities to give a simple proof of the Division Conjecture, using methods similar to those in our proof of Theorem \ref{mainthm:Lojasiewicz_gradient_inequality}. Atiyah notes \cite[p. 145]{Atiyah_1970} that Bernstein and Gelfand independently proved the Division Conjecture using related ideas in \cite{Bernstein_Gelfand_1969}. The only article that is firmly in the literature on {\L}ojasiewicz inequalities that cites Atiyah is due to Bivi{\`a}--Ausina and Fukui \cite{Bivia-Ausina_Fukui_2016}.

\subsection{Main results}
\label{subsec:Main_results}
We let $\KK=\RR$ or $\CC$ and state the main results to be proved in this article, categorized according to whether or not their proofs appeal to resolution of singularities.

\subsubsection{Gradient inequality using resolution of singularities}
\label{subsubsec:Main_results_resolution_singularities}
We begin with the fundamental

\begin{mainthm}[{\L}ojasiewicz gradient inequality for an analytic function]
\label{mainthm:Lojasiewicz_gradient_inequality}
(See {\L}ojasiewicz \cite[Proposition 1, p. 92 (67)]{Lojasiewicz_1965}.)
Let $d \geq 1$ be an integer, $U \subset \KK^d$ be an open subset, and $\sE:U\to\KK$ be an analytic function. If $x_\infty\in U$ is a point such that $\sE'(x_\infty) = 0$, then there are constants $C_0 \in (0, \infty)$, and $\sigma_0 \in (0,1]$, and $\theta \in [1/2,1)$ such that the differential map, $\sE':U \to \KK^{d*}$, obeys
\begin{equation}
\label{eq:Lojasiewicz_gradient_inequality}
\|\sE'(x)\|_{\KK^{d*}} \geq C_0|\sE(x) - \sE(x_\infty)|^\theta,
\quad\text{for all } x \in B_{\sigma_0}(x_\infty),
\end{equation}
where $\KK^{d*} = (\KK^d)^*$, the dual space of $\KK^d$ and $B_{\sigma_0}(x_\infty) := \{x \in \KK^d: \|x-x_\infty\|_{\KK^d} < \sigma_0 \} \subset U$.
\end{mainthm}

By definition, the \emph{{\L}ojasiewicz exponent} $\theta$ of a $C^1$ function $\sE$ at a point $x_\infty$ in its domain is the \emph{smallest} $\theta \geq 0$ such that the inequality \eqref{eq:Lojasiewicz_gradient_inequality} holds for some positive constant $C_0$ and all $x$ in an open neighborhood of $x_\infty$.

Theorem \ref{mainthm:Lojasiewicz_gradient_inequality} was stated by {\L}ojasiewicz in \cite[Theorem 4]{Lojasiewicz_1963} and proved by him as \cite[Proposition 1, p. 92]{Lojasiewicz_1965}; see also {\L}ojasiewicz \cite[p. 1592]{Lojasiewicz_1993}. Bierstone and Milman provided simplified proofs as \cite[Proposition 6.8]{BierstoneMilman} and
\cite[Theorem 2.7]{Bierstone_Milman_1997}. Their strategy in \cite{BierstoneMilman} is to first prove a {\L}ojasiewicz inequality \cite[Theorem 6.4]{BierstoneMilman} of the form
\begin{equation}
\label{eq:Lojasiewicz_division_inequality_pair_functions}
|g(x)| \geq C|f(x)|^\lambda, \quad\text{for all } x \in B_\sigma,
\end{equation}
where $f$ and $g$ are subanalytic functions on an open neighborhood $U\subset\RR^d$ of the origin such that $g^{-1}(0) \subset f^{-1}(0)$ and $B_\sigma \subset U$ and $\lambda \in (0,\infty)$. They then deduce a {\L}ojasiewicz gradient inequality \cite[Theorem 6.8]{BierstoneMilman} for a real analytic function $f$ with $f'(0)=0$,
\begin{equation}
\label{eq:Lojasiewicz_gradient_inequality_pair_functions}
\|f'(x)\|_{\RR^{d*}} \geq C|f(x)|^\nu, \quad\text{for all } x \in B_\sigma,
\end{equation}
with $\nu \in (0,1)$ by choosing $g = \|f'\|_{\RR^{d*}}$. In \cite[Theorem 2.5]{Bierstone_Milman_1997}, the authors establish \eqref{eq:Lojasiewicz_division_inequality_pair_functions} for a pair of (real or complex) analytic functions by using resolution of singularities to reduce to the case that the ideal in the ring of analytic functions, $\sO_X$, generated by $fg$ has simple normal crossings. In \cite[Theorem 2.7]{Bierstone_Milman_1997}, they then obtain \eqref{eq:Lojasiewicz_gradient_inequality_pair_functions} for an analytic function $f$ with $f(0)=0$ and $f'(0)=0$ by choosing $g = \|f'\|_{\KK^{d*}}^2$ and applying \eqref{eq:Lojasiewicz_division_inequality_pair_functions} to the pair of functions $f^2$ (replacing $g$) and $f^2/g$ (replacing $f$) and proving that $f^{-1}(0) \subset (f^2/g)^{-1}(0)$ and $\nu = 1/\lambda\in (0,1)$, after employing resolution of singularities to the ideal $fg\sO_X$.

Our more direct proof of Theorem \ref{mainthm:Lojasiewicz_gradient_inequality} makes it clear that one always has $\theta\geq 1/2$, whereas previous proofs only gave $\theta \in (0,1)$. For applications to geometric analysis and topology, it is essential to have $\theta < 1$, with $\theta=1/2$ being the optimal exponent, corresponding to exponential convergence for the negative gradient flow defined by $\sE$. In particular, we have:

\begin{maincor}[Characterization of the optimal exponent and Morse--Bott condition]
\label{maincor:Characterization_optimal_exponent_Morse-Bott_condition}
Assume the hypotheses of Theorem \ref{mainthm:Lojasiewicz_gradient_inequality}
and that $x_\infty$ is the origin. If $\theta=1/2$ then, after possibly shrinking $U$, there are an open neighborhood of the origin, $\widetilde{U} \subset \KK^d$, and an analytic map, $\pi:\widetilde{U}\to U$, such that $\pi$ is an analytic diffeomorphism on the complement of a coordinate hyperplane or the
union of two coordinate hyperplanes and $\pi^*\sE$ is Morse--Bott at the origin in the sense of Definition \ref{defn:Morse-Bott_function}.
\end{maincor}

See the author's \cite[Theorem 3]{Feehan_yang_mills_gradient_flow_v4} for the statement and proof of a very general convergence-rate result for an abstract gradient flow on a Banach space defined by an analytic function obeying a {\L}ojasiewicz--Simon gradient inequality with exponent $\theta \in [1/2,1)$ and for previous versions of related convergence-rate results, see Chill, Haraux, and Jendoubi \cite[Theorem 2]{Chill_Haraux_Jendoubi_2009}, Haraux, Jendoubi, and Kavian \cite[Propositions 3.1 and 3.4]{Haraux_Jendoubi_Kavian_2003}, Huang \cite[Theorem 3.4.8]{Huang_2006}, and R\r{a}de \cite[Proposition 7.4]{Rade_1992}. Convergence-rate results related to \cite[Theorem 3]{Feehan_yang_mills_gradient_flow_v4} are implicit in Adams and Simon \cite{Adams_Simon_1988} and Simon \cite{Simon_1983, Simon_1985, Simon_1996}, although we cannot find an explicit statement like this in those references.

\subsubsection{Gradient inequality without using resolution of singularities}
\label{subsubsec:Main_results_without_resolution_singularities}
The proof of Theorem \ref{mainthm:Lojasiewicz_gradient_inequality} in full generality provided in this article employs embedded resolution of singularities (partly following Bierstone and Milman \cite[Section 2]{Bierstone_Milman_1997}), but there are several weaker gradient inequalities that can be proved by far more elementary methods and those provide insight to applications in geometric analysis. We now describe several results of this kind. For example, when the function $\sE$ in Theorem \ref{mainthm:Lojasiewicz_gradient_inequality} is $C^2$ (respectively, $C^N$ with $N\geq 2$) and Morse--Bott (respectively, Morse--Bott of order $N$), rather than an arbitrary analytic function, one obtains the {\L}ojasiewicz gradient inequality with exponent $\theta=1/2$ (respectively, $\theta=1-1/N$) as a consequence of the Mean Value Theorem (respectively, Taylor Theorem): see Theorems \ref{thm:Lojasiewicz_gradient_inequality_Morse-Bott} and \ref{thm:Lojasiewicz_gradient_inequality_generalized_Morse-Bott_function}. We refer the reader to Section \ref{sec:Lojasiewicz-Simon_gradient_inequality_Morse-Bott} for a discussion of the Morse--Bott condition and some its generalizations, together with the statements and proofs of Theorems \ref{thm:Lojasiewicz_gradient_inequality_Morse-Bott} and \ref{thm:Lojasiewicz_gradient_inequality_generalized_Morse-Bott_function}.

A first reading of the proof of Theorem \ref{thm:Lojasiewicz_gradient_inequality_generalized_Morse-Bott_function}, which is based on a direct application of the Taylor Theorem, might suggest that it would extend to the case where $\sE$ is an analytic function and $U\cap \Crit\sE$ is an arbitrary analytic subvariety. However, one finds that this is a more difficult strategy to develop than one might naively expect. Instead, as a stepping stone towards Theorem \ref{mainthm:Lojasiewicz_gradient_inequality}, we shall first establish a special case that holds for a class of $C^1$ functions. By analogy with Collins, Greenleaf, and Pramanik \cite[Definition 2.5]{Collins_Greenleaf_Pramanik_2013}, we make the

\begin{defn}[Function with simple normal crossings]
\label{defn:C1_function_normal_crossings}
A $C^1$ function $f:U\to\KK$ on an open neighborhood of the origin in $\KK^d$ has \emph{simple normal crossings} if
\begin{equation}
\label{eq:C1_function_normal_crossings}
f(x) = x_1^{n_1}\cdots x_d^{n_d}f_0(x), \quad\text{for all } x = (x_1,\ldots,x_d) \in U,
\end{equation}
where $n_i \in \ZZ\cap [0,\infty)$ and $f_0$ is a $C^1$ function such that $f_0(0)\neq 0$ and $N=\sum_{i=1}^dn_i$ is the \emph{total degree} of the monomial $x_1^{n_1}\cdots x_d^{n_d}$.
\end{defn}

See Sections \ref{subsec:Normal_crossing_divisors} and \ref{subsec:Ideals_smiple_normal_crossing_divisors} for a review of normal crossings and simple normal crossings divisors in (real or complex) analytic geometry.

\begin{defn}[Morse--Bott function]
\label{defn:Morse-Bott_function}
Let $d\geq 1$ be an integer, $U \subset \KK^d$ be an open subset, $\sE:U\to\KK$ be a $C^2$ function, and $\Crit\sE := \{x\in U:\sE'(x) = 0\}$. We say that $\sE$ is \emph{Morse--Bott at a point} $x_\infty \in \Crit\sE$ if
\begin{inparaenum}[\itshape a\upshape)]
\item\label{item:Morse-Bott_function_critical_set_submanifold} $\Crit\sE$ is a $C^2$ submanifold of $U$, and
\item\label{item:Morse-Bott_function_tangent_space_critical_set_equals_kernel_hessian} $T_{x_\infty}\Crit\sE = \Ker\sE''(x_\infty)$ when $\sE''(x_\infty)$ is considered as an operator in $\Hom_\KK(\KK^d,\KK^{d*})$,
\end{inparaenum}
where $T_x\Crit\sE$ is the tangent space to $\Crit\sE$ at a point $x \in \Crit\sE$.
\end{defn}

In applications to topology (see, for example, Austin and Braam \cite[Section 3.1]{Austin_Braam_1995} for equivariant Floer cohomology and Bott \cite{Bott_1959} for the Periodicity Theorem), our local Definition \ref{defn:Morse-Bott_function} is often augmented by conditions that $\Crit\sE$ be compact, as in Bott \cite[Definition, p. 248]{Bott_1954}, or compact and connected as in Nicolaescu
\cite[Definition 2.41]{Nicolaescu_morse_theory}, and that $T_x\Crit\sE = \Ker\sE''(x)$ for all $x \in \Crit\sE$.

\begin{mainthm}[{\L}ojasiewicz gradient inequality for a $C^1$ function with simple normal crossings and characterization of the optimal exponent and Morse--Bott condition]
\label{mainthm:Lojasiewicz_gradient_inequality_normal_crossing_divisor}
Let $d \geq 2$ be an integer, $U \subset \KK^d$ be an open neighborhood of the origin, and $\sE:U \to \KK$ be a $C^1$ function with simple normal crossings. If $\sE'(0)=0$, then the followibng hold.
\begin{enumerate}
\item There are constants $C_0 \in (0,\infty)$ and $\sigma \in (0,1]$ such that
\begin{equation}
\label{eq:Lojasiewicz_gradient_inequality_normal_crossing_divisor}
\|\sE'(x)\|_{\KK^{d*}} \geq C_0|\sE(x)|^\theta, \quad\text{for all } x \in B_\sigma,
\end{equation}
where $\theta = 1-1/N \in [1/2,1)$ and $N=\sum_{i=1}^dn_i$ is the total degree of the monomial in the expression \eqref{eq:C1_function_normal_crossings} for $\sE$.
\item If $c$ is the number of exponents $n_i\geq 1$ for $i=1,\ldots,d$, then $c\geq 2$ or $c=1$ and (after relabeling coordinates) $n_1\geq 2$.
\item One has $\theta=1/2$ if and only if $c=2$ and (after relabeling coordinates) $n_1=n_2=1$ or $c=1$ and $n_1=2$.
\item If $\theta=1/2$ and $\sE$ is $C^2$, then $\sE$ is Morse--Bott on $B_\sigma$.
\end{enumerate}
\end{mainthm}

\begin{rmk}[Geometry of the critical set]
\label{rmk:Geometry_critical_set}
Theorem \ref{thm:Lojasiewicz_gradient_inequality_Morse-Bott} shows that, when $\sE$ is Morse--Bott and so its critical set is a smooth submanifold, then its {\L}ojasiewicz exponent $\theta$ is equal to $1/2$. Conversely, when $\theta=1/2$,
Theorem \ref{mainthm:Lojasiewicz_gradient_inequality_normal_crossing_divisor} implies that $B_\sigma\cap\Crit\sE = \{x_1=0\}\cap B_\sigma$ or $\{x_1=x_2=0\}\cap B_\sigma$, a codimension-one or codimension-two smooth submanifold of $B_\sigma$. Theorem \ref{mainthm:Lojasiewicz_gradient_inequality} is proved by applying resolution of singularities to an ideal defined by an arbitrary analytic function $\sE$ and applying Theorem \ref{mainthm:Lojasiewicz_gradient_inequality_normal_crossing_divisor} to the resulting monomial (the product of $x_1^{n_1}\cdots x_d^{n_c}$ and a non-vanishing analytic function). Consequently, if $\theta=1/2$ then there is a constraint on the nature of the singularities in the critical set of $\sE$. Our proof of Theorem \ref{mainthm:Lojasiewicz_gradient_inequality} shows that application of resolution of singularities does not change the {\L}ojasiewicz exponent and so it is of interest to try to characterize the class of analytic functions with $\theta=1/2$, a topic that we explore in Feehan \cite{Feehan_lojasiewicz_inequality_all_dimensions_morse-bott}. As noted in our Introduction, the problem of computing or estimating {\L}ojasiewicz exponents remains a topic of active research.
\end{rmk}

Our proof of Theorem \ref{mainthm:Lojasiewicz_gradient_inequality_normal_crossing_divisor} is a direct coordinate-based alternative to an argument due to Bierstone and Milman \cite[Section 2]{Bierstone_Milman_1997} and relies only on the Generalized Young Inequality \eqref{eq:Generalized_Young_inequality} (see Remark \eqref{rmk:Generalized_Young_inequality}). We are grateful to Alain Haraux for pointing out that the value for $\theta$ in previous versions of this article could be improved to the value now stated in Theorem \ref{mainthm:Lojasiewicz_gradient_inequality_normal_crossing_divisor} and for alerting us to his \cite[Theorem 3.1]{Haraux_2005}. His result is more closely related to Theorem \ref{mainthm:Lojasiewicz_gradient_inequality_normal_crossing_divisor} than we had realized (it assumes $f_0=1$ in the expression \eqref{eq:C1_function_normal_crossings}) and we were unaware that his proof also used the Generalized Young Inequality.

\subsubsection{Consequences of the gradient inequality}
\label{subsubsec:Consequences_Lojasiewicz_gradient_inequality}
Regardless of how proved, the gradient inequality \eqref{eq:Lojasiewicz_gradient_inequality} easily yields two useful corollaries. Note that if $\sE(x)$ is differentiable at $x=x_0$ and $\sE(x_0)=0$, then $\sE(x)^2$ has a critical point at $x=x_0$. We say that a subset $A\subset\KK^d$ is \emph{$C^k$-arc connected} if any two points in $A$ can be joined by a $C^k$ curve, where $k \in \ZZ\cap [0,\infty)$ or $k=\omega$ (analytic), and \emph{locally $C^k$-arc connected} if for every point $x\in A$ has an open neighborhood $U\subset \KK^d$ such that $U\cap A$ is $C^k$-arc connected.

\begin{maincor}[{\L}ojasiewicz distance inequalities]
\label{maincor:Lojasiewicz_distance_inequality}
(Compare {\L}ojasiewicz \cite[Theorem 2, p. 85 (62)]{Lojasiewicz_1965}.)
Let $d \geq 1$ be an integer, $U \subset \RR^d$ be an open neighborhood of the origin and $\sE:U\to\RR$ be a $C^1$ function.
\begin{enumerate}
\item \emph{(Distance to the critical and zero sets.)}
\label{item:Distance_critical_set}  
If $\sE(0)=0$ and $\sE'(0)=0$ and $\sE \geq 0$ on $U$ and $\sE$ obeys the {\L}ojasiewicz gradient inequality \eqref{eq:Lojasiewicz_gradient_inequality} near the origin, then there are constants $C_1 \in (0, \infty)$, and $\delta \in (0,\sigma/4]$, and $\alpha=1/(1-\theta) \in [2,\infty)$ such that
\begin{equation}
\label{eq:Lojasiewicz_distance_inequality_critical_set}
\sE(x) \geq C_1\dist_{\RR^d}(x, B_\sigma\cap \Crit\sE)^\alpha, \quad\text{for all } x \in B_\delta,
\end{equation}
where $\dist_{\RR^d}(x,A) := \inf\{\|x-a\|_{\RR^d}: a\in A\}$, for any point $x\in\RR^d$ and subset $A \subset \RR^d$. If in addition $B_\sigma\cap \Crit\sE \subset B_\sigma\cap \Zero\sE$, then
\begin{equation}
\label{eq:Lojasiewicz_distance_inequality_zero_set}
\sE(x) \geq C_1\dist_{\RR^d}(x, B_\sigma\cap \Zero\sE)^\alpha, \quad\text{for all } x \in B_\delta,
\end{equation}
where $\Zero\sE := \{x\in U:\sE(x)=0\}$.

\item \emph{(Distance to the zero set.)} 
\label{item:Distance_zero_noncritical_set}
If $\sE(0)=0$ and $\sE^2$ (in place of $\sE$) obeys the {\L}ojasiewicz gradient inequality \eqref{eq:Lojasiewicz_gradient_inequality} near the origin and $B_\sigma\cap \Crit\sE^2 \subset B_\sigma\cap \Zero\sE$, then there are constants $C_2 \in (0, \infty)$, and $\delta \in (0,\sigma/4]$, and $\beta =1/(2(1-\theta))\in [1,\infty)$ such that
\begin{equation}
\label{eq:Lojasiewicz_distance_inequality_zero_noncritical_set}
|\sE(x)| \geq C_2\dist_{\RR^d}(x, B_\sigma\cap \Zero\sE)^\beta, \quad\text{for all } x \in B_\delta.
\end{equation}
\end{enumerate}
\end{maincor}

\begin{rmk}[Analytic functions obey the hypotheses of Corollary \ref{maincor:Lojasiewicz_distance_inequality}]
\label{rmk:Lojasiewicz_distance_inequality_analytic}
If $\sE$ is analytic, then the hypotheses in Corollary \ref{maincor:Lojasiewicz_distance_inequality} that $\sE$ or $\sE^2$ obey \eqref{eq:Lojasiewicz_gradient_inequality} are implied by Theorem \ref{mainthm:Lojasiewicz_gradient_inequality}.

Moreover, if $\sE$ is analytic, then $\Crit\sE$ and $\Crit\sE^2$ are analytic subvarieties of $U$ and thus locally connected by \cite[Corollary 2.7 (3)]{BierstoneMilman} and hence locally $C^0$-arc connected by \cite[Exercise 29F]{Kowalsky_topological_spaces}. By Gabri{\'e}lov \cite[p. 283]{Gabrielov_1968}, analytic subvarieties of $U$ are locally analytic-arc connected and so $\Crit\sE$ and $\Crit\sE^2$ are locally $C^1$-arc connected by \cite[p. 283]{Gabrielov_1968} when $\sE$ is analytic. In particular, if $\Crit\sE$ is locally $C^1$-arc connected, then $B_\sigma\cap \Crit\sE \subset B_\sigma\cap \Zero\sE$, as assumed in the second half of Item \eqref{item:Distance_critical_set}; if $\Crit\sE^2$ is locally $C^1$-arc connected, then $B_\sigma\cap \Crit\sE^2 \subset B_\sigma\cap \Zero\sE$, as assumed in Item \eqref{item:Distance_zero_noncritical_set}.
\end{rmk}

When $\sE$ is analytic, Item \eqref{item:Distance_zero_noncritical_set} in Corollary \ref{maincor:Lojasiewicz_distance_inequality} was stated by {\L}ojasiewicz in \cite[Corollary, p. 88]{Lojasiewicz_1963} and proved by him in \cite[Theorem 17, p. 124]{Lojasiewicz_1959}, \cite[Theorem 17, p. 40]{Lojasiewicz_1961}; it was restated and proved by him as \cite[Theorem 2, p. 85 (62)]{Lojasiewicz_1965}. Simplified proofs of Item \eqref{item:Distance_zero_noncritical_set} in Corollary \ref{maincor:Lojasiewicz_distance_inequality} were provided by Bierstone and Milman as \cite[Theorem 6.4 and Remark 6.5]{BierstoneMilman} and
\cite[Theorem 2.8]{Bierstone_Milman_1997}. When $\sE$ is a polynomial on $\RR^d$, Corollary \ref{maincor:Lojasiewicz_distance_inequality} is due to H\"ormander \cite[Lemma 1]{Hormander_1958}. The next result is obtained by combining Theorem \ref{mainthm:Lojasiewicz_gradient_inequality} and Item \eqref{item:Distance_critical_set} in Corollary \ref{maincor:Lojasiewicz_distance_inequality}.

\begin{maincor}[{\L}ojasiewicz gradient-distance inequality for a non-negative function]
\label{maincor:Lojasiewicz_gradient-distance_inequality}
Let $d \geq 1$ be an integer, $U \subset \RR^d$ be an open neighborhood of a point $x_\infty$, and $\sE:U\to\RR$ be a
$C^1$ function. If $\sE'(x_\infty)=0$ and $\sE\geq 0$ on $U$ and $\sE$ obeys the {\L}ojasiewicz gradient inequality \eqref{eq:Lojasiewicz_gradient_inequality} near $x_\infty$, then there are constants $C_2 \in (0, \infty)$, and $\delta \in (0,\sigma/4]$, and $\mu=\theta/(1-\theta) \in [1,\infty)$ such that 
\begin{equation}
\label{eq:Lojasiewicz_gradient-distance_inequality}
\|\sE'(x)\|_{\RR^{d*}} \geq C_2\dist(x, B_\sigma\cap \Crit\sE)^\mu,
\quad\text{for all } x \in B_\delta(x_\infty).
\end{equation}
\end{maincor}

When $\sE$ is analytic, the hypothesis in Corollary \ref{maincor:Lojasiewicz_gradient-distance_inequality} that $\sE\geq 0$ on $U$ can be relaxed.

\begin{maincor}[{\L}ojasiewicz gradient-distance inequality for an analytic function]
\label{maincor:Lojasiewicz_gradient-distance_inequality_analytic}
Let $d \geq 1$ be an integer, $U \subset \RR^d$ be an open neighborhood of a point $x_\infty$, and $\sE:U\to\RR$ be an analytic function. If $\sE'(x_\infty)=0$, then there are constants $C_3 \in (0, \infty)$, and $\sigma_1 \in (0,1]$, and $\delta_1 \in (0,\sigma_1/4]$, and
$\gamma \in [1/2,\infty)$ such that 
\begin{equation}
\label{eq:Lojasiewicz_gradient-distance_inequality_analytic}
\|\sE'(x)\|_{\RR^{d*}} \geq C_3\dist_{\RR^d}(x, B_{\sigma_1}\cap \Crit\sE)^\gamma,
\quad\text{for all } x \in B_{\delta_1}(x_\infty).
\end{equation}
\end{maincor}

The inequality \eqref{eq:Lojasiewicz_gradient-distance_inequality_analytic} is stated by Simon in \cite[Equation (2.3)]{Simon_1983} and attributed by him to {\L}ojasiewicz \cite{Lojasiewicz_1965}.

\subsubsection{Counterexamples}
\label{subsubsec:Counterexample_Lojasiewicz_gradient_inequality}
It is known but worth remembering that the {\L}ojasiewicz gradient inequality fails in general for functions that are smooth but not analytic. For example, De Lellis \cite{DeLellis_Encyc-Math_LS_inequality} notes that when $d=1$, then the function
\[
\sE(x)
=
\begin{cases}
e^{1/|x|}, &x \neq 0,
\\
0, &x=0,
\end{cases}
\]
is $C^\infty$ on $\RR$ with $\Crit\sE = \{0\}$ but that the inequality \eqref{eq:Lojasiewicz_gradient_inequality} fails on any open neighborhood of the origin. When $d=2$ and $\KK=\RR$, Haraux shows in \cite[Proposition 5.2]{Haraux_2005} that for the $C^1$ function,
\[
\sE(x,y)
=
\begin{cases}
(x^2+y^2)e^{-(x^2+y^2)/x^2}, &x \neq 0,
\\
0, &x=0,
\end{cases}
\]
the inequality \eqref{eq:Lojasiewicz_gradient_inequality} fails on any neighborhood of the origin. Moreover, failure of a smooth function to satisfy the {\L}ojasiewicz gradient inequality may result in
non-convergence of its negative gradient flow: see Haraux \cite[Remark 5.5]{Haraux_2005} (citing Palis and de Melo \cite{Palis_Welington_geometric_theory_dynamical_systems}), Haraux and Jendoubi \cite[Section 12.8]{Haraux_Jendoubi_convergence_dissipative_autonomous_systems}, and Lerman \cite{Lerman_2005} (citing \cite[p. 14]{Palis_Welington_geometric_theory_dynamical_systems}).

\subsection{Outline}
\label{subsec:Outline}
We begin in Section \ref{sec:Lojasiewicz-Simon_gradient_inequality_Morse-Bott} with elementary proofs of the {\L}ojasiewicz gradient inequality for functions that are Morse--Bott (Theorem \ref{thm:Lojasiewicz_gradient_inequality_Morse-Bott}) or generalized Morse--Bott (Theorem \ref{thm:Lojasiewicz_gradient_inequality_generalized_Morse-Bott_function}). In Section \ref{sec:Lojasiewicz_gradient_inequality_critical_normal_crossing_divisor}, we establish the {\L}ojasiewicz gradient inequality (Theorem \ref{mainthm:Lojasiewicz_gradient_inequality_normal_crossing_divisor}) for $C^1$ functions with simple normal crossings. In Section \ref{sec:Resolution_singularities_and_Lojasiewicz_gradient_inequality}, we review the resolution of singularities for analytic varieties (Theorem \ref{thm:Monomialization_ideal_sheaf}) and apply that and Theorem \ref{mainthm:Lojasiewicz_gradient_inequality_normal_crossing_divisor} to prove the {\L}ojasiewicz gradient inequality for an arbitrary analytic function (Theorem \ref{mainthm:Lojasiewicz_gradient_inequality}). Finally, in Section \ref{sec:Lojasiewicz_distance_inequalities} we deduce Corollaries \ref{maincor:Lojasiewicz_distance_inequality}, \ref{maincor:Lojasiewicz_gradient-distance_inequality}, and \ref{maincor:Lojasiewicz_gradient-distance_inequality_analytic} from the gradient inequality \eqref{eq:Lojasiewicz_gradient_inequality}. Appendix \ref{sec:Resolution_singularities_cusp_curve} illustrates the application of resolution of singularities (Theorem \ref{thm:Monomialization_ideal_sheaf}) to achieve the required monomialization in the case of a simple example, namely the cusp curve. 

\subsection{Acknowledgments}
\label{subsec:Acknowledgments}
I thank Kenji Fukaya for his description of an example that alerted me to an error in an entirely different approach described in an early draft of this article. I am grateful to Lev Borisov, Tristan Collins, Antonella Grassi, Michael Greenblatt, Johan de Jong, Peter Kronheimer, Claude LeBrun, Tom Mrowka, Graeme Wilkin, and Jarek W{\l}odarczyk for helpful communications or discussions, to Carles Bivi{\`a}-Ausina, Santiago Encinas, Alain Haraux and Dennis Sullivan for many helpful comments and suggestions, to Manousos Maridakis for engaging conversations regarding {\L}ojasiewicz inequalities, to Toby Colding and Mark Goresky for their interest in this work, and to our librarian at Rutgers, Mei-Ling Lo, for locating difficult-to-find articles on my behalf. I am indebted to Tom\'{a}\v{s} B\'{a}rta for helpful discussions and especially for pointing out an example (see Remark \ref{defn:Generalized_Morse-Bott_function}) showing that a previous version of 
my Definition \ref{defn:Generalized_Morse-Bott_function} in an older version of this article was incomplete.
I am grateful to the National Science Foundation (grant DMS-1510064) for their support and to the Simons Center for Geometry and Physics, Stony Brook, the Dublin Institute for Advanced Studies, the Institut des Hautes {\'E}tudes Scientifiques, Bures-sur-Yvette, and the Institute for Advanced Studies, Princeton, for their hospitality and support during the preparation of this article. Lastly, I thank the anonymous referees for their careful reading of our manuscript and for their suggestions.

\section{{\L}ojasiewicz gradient inequalities for generalized Morse--Bott functions}
\label{sec:Lojasiewicz-Simon_gradient_inequality_Morse-Bott}
In this section, we adapt our previous proof in \cite{Feehan_lojasiewicz_inequality_ground_state} of the {\L}ojasiewicz inequalities for Morse--Bott functions on Banach spaces
\cite[Theorem 3]{Feehan_lojasiewicz_inequality_ground_state} (restated here as Theorem \ref{thm:Lojasiewicz_gradient_inequality_Morse-Bott} for the case of Euclidean spaces) to prove the {\L}ojasiewicz gradient inequality for generalized Morse--Bott functions, namely Theorem \ref{thm:Lojasiewicz_gradient_inequality_generalized_Morse-Bott_function}; our \cite[Theorem 3]{Feehan_lojasiewicz_inequality_ground_state} improves upon \cite[Theorems 3 and 4]{Feehan_Maridakis_Lojasiewicz-Simon_Banach} and has a simpler proof. Theorem \ref{thm:Lojasiewicz_gradient_inequality_Morse-Bott} was proved by Simon \cite[Lemma 3.13.1]{Simon_1996} (for a harmonic map energy function on a Banach space of $C^{2,\alpha}$ sections of a Riemannian vector bundle), Haraux and Jendoubi \cite[Theorem 2.1]{Haraux_Jendoubi_2007} (for functions on abstract Hilbert spaces), and in greater generality by Chill in \cite[Corollary 3.12]{Chill_2003} (for functions on abstract Banach spaces); a more elementary version was proved by Huang as \cite[Proposition 2.7.1]{Huang_2006} (for functions on abstract Banach spaces). These authors do not use Morse--Bott terminology but their hypotheses imply this condition --- directly in the case of Haraux and Jendoubi and Chill and by a remark due to Simon in \cite[p. 80] {Simon_1996} that his integrability condition \cite[Equation (iii), p. 79]{Simon_1996} is equivalent to a restatement of the Morse--Bott condition. See Feehan \cite{Feehan_lojasiewicz_inequality_all_dimensions_morse-bott} for further discussion of the relationship between definitions of integrability, such as those described by Adams and Simon \cite{Adams_Simon_1988}, and the Morse--Bott condition.

\subsection{Morse--Bott and generalized Morse--Bott functions}
\label{sec:Morse-Bott_function}
We begin with a well-known result.

\begin{thm}[{\L}ojasiewicz gradient inequality for a Morse--Bott function on Euclidean space]
\label{thm:Lojasiewicz_gradient_inequality_Morse-Bott}
Let $d \geq 1$ be an integer and $U \subset \KK^d$ an open subset. If $\sE:U \to \KK$ is a Morse--Bott function, then there are constants $C_0 \in (0,\infty)$ and $\sigma_0 \in (0,1]$ such that
\begin{equation}
\label{eq:Lojasiewicz-Simon_gradient_inequality_Morse-Bott}
\|\sE'(x)\|_{\KK^{d*}} \geq C_0|\sE(x) - \sE(x_\infty)|^{1/2},
\quad\text{for all } x \in B_{\sigma_0}(x_\infty).
\end{equation}
\end{thm}

Theorem \ref{thm:Lojasiewicz_gradient_inequality_Morse-Bott} is a special case of Feehan \cite[Theorem 3]{Feehan_lojasiewicz_inequality_ground_state} and Feehan and Maridakis \cite[Theorems 3 and 4]{Feehan_Maridakis_Lojasiewicz-Simon_Banach}, where the case of a Morse--Bott function on a Banach space is considered. Even when $\sE$ is a Morse--Bott function on a Banach space, the proof of the corresponding {\L}ojasiewicz gradient inequality \cite[Theorem 3]{Feehan_lojasiewicz_inequality_ground_state} still readily follows from the Mean Value Theorem (see \cite[Section 4]{Feehan_lojasiewicz_inequality_ground_state}) in the presence of a few additional technical hypotheses specific to the infinite-dimensional setting. 

\begin{rmk}[On the proof of Theorem \ref{thm:Lojasiewicz_gradient_inequality_Morse-Bott}]
\label{rmk:Lojasiewicz_gradient_inequality_Morse-Bott_on_proof}
The conclusion of Theorem \ref{thm:Lojasiewicz_gradient_inequality_Morse-Bott} is a simple consequence of the Morse--Bott Lemma (see Banyaga and Hurtubise \cite[Theorem 2]{Banyaga_Hurtubise_2004}, Nicolaescu \cite[Proposition 2.42]{Nicolaescu_morse_theory}, or Feehan \cite{Feehan_lojasiewicz_inequality_all_dimensions_morse-bott}). However, the proof of the Morse--Bott Lemma itself (especially for Morse--Bott functions that are at most $C^2$) requires care. In contrast, our proof of Theorem \ref{thm:Lojasiewicz_gradient_inequality_Morse-Bott} --- given as the proof of \cite[Theorem 3]{Feehan_lojasiewicz_inequality_ground_state} in the infinite-dimensional case --- is direct and elementary and avoids appealing to the Morse--Bott Lemma.
\end{rmk}

\begin{defn}[Generalized Morse--Bott function]
\label{defn:Generalized_Morse-Bott_function}
Let $d\geq 1$ and $N \geq 2$ be integers, $U \subset \KK^d$ be an open subset, and $\sE:U\to\KK$ be a $C^N$ function. We call $\sE$ a \emph{generalized Morse--Bott function of order $N$ at a point $x_\infty \in \Crit\sE$} if
\begin{inparaenum}[(\itshape a\upshape)]
\item\label{item:Generalized_Morse-Bott_function_submanifold} $\Crit\sE$ is a $C^N$ submanifold of $U$, 
\item\label{item:Generalized_Morse-Bott_function_lower_order_zero} $\sE^{(n)}(x) = 0$ for all $x \in \Crit\sE$ and $1 \leq n \leq N-1$, and
\item\label{item:Generalized_Morse-Bott_function_coercive_N_order_derivative} $\sE^{(N)}(x_\infty)\xi^N \neq 0$ for all nonzero $\xi \in T_{x_\infty}^\perp\Crit\sE$, where $T_{x_\infty}^\perp\Crit\sE$ is the orthogonal complement of
  $T_{x_\infty}\Crit\sE$ in $\KK^d$.
\end{inparaenum}
\end{defn}

For example, if $N \geq 2$ and $f(x,y) = x^N$ then $f:\KK^2\to\KK$ is a generalized Morse--Bott function of order $N$. The analogous definition of a generalized Morse function is stated, for example, by Rothe \cite[Definition 2.6]{Rothe_1973} and Kuiper \cite[p. 202, Corollary]{Kuiper_1972}. While Definition \ref{defn:Generalized_Morse-Bott_function} is valid when $N=2$, the conditions are then more restrictive than those of Definition \ref{defn:Morse-Bott_function}.

\begin{thm}[{\L}ojasiewicz gradient inequality for a generalized Morse--Bott function on Euclidean space]
\label{thm:Lojasiewicz_gradient_inequality_generalized_Morse-Bott_function}
Let $d \geq 1$ and $N\geq 2$ be integers and $U \subset \KK^d$ be an open neighborhood. If $\sE:U \to \KK$ is a generalized Morse--Bott function of order $N$ at a point $x_\infty \in \Crit\sE$, then there are constants $C_0 \in (0,\infty)$ and $\sigma_0 \in (0,1]$ such that
\begin{equation}
\label{eq:Lojasiewicz_gradient_inequality_N-Morse-Bott_function}
\|\sE'(x)\|_{\KK^{d*}} \geq C_0|\sE(x) - \sE(x_\infty)|^{1-1/N}, \quad\text{for all } x \in B_{\sigma_0}(x_\infty).
\end{equation}
\end{thm}

As Definition \ref{defn:Generalized_Morse-Bott_function} suggests, the proof of Theorem \ref{thm:Lojasiewicz_gradient_inequality_generalized_Morse-Bott_function} should generalize to the setting of functions on Banach spaces, as in  \cite[Theorem 3]{Feehan_lojasiewicz_inequality_ground_state} for the case of Morse--Bott functions.

\begin{rmk}[Comparison of Theorem \ref{thm:Lojasiewicz_gradient_inequality_Morse-Bott} and Theorem \ref{thm:Lojasiewicz_gradient_inequality_generalized_Morse-Bott_function} when $N=2$]
\label{rmk:Comparison_Morse-Bott_and_generalized_Morse-Bott_when_N_is_2}  
While Theorem \ref{thm:Lojasiewicz_gradient_inequality_generalized_Morse-Bott_function} holds when $N=2$, Theorem \ref{thm:Lojasiewicz_gradient_inequality_Morse-Bott} is a stronger result since condition \eqref{item:Morse-Bott_function_tangent_space_critical_set_equals_kernel_hessian} in Definition \ref{defn:Morse-Bott_function}, which is equivalent to the condition that $\sE''(x_\infty) \in \End_\KK(T_{x_\infty}^\perp\Crit\sE)$ be invertible, is weaker than the coercivity condition \eqref{item:Generalized_Morse-Bott_function_coercive_N_order_derivative} in Definition \ref{defn:Generalized_Morse-Bott_function}, namely, that $\sE''(x_\infty)\xi^2\neq 0$ for all non-zero $\xi \in T_{x_\infty}^\perp\Crit\sE$.
\end{rmk}

\begin{rmk}[On Definition \ref{defn:Generalized_Morse-Bott_function} and the hypotheses of Theorem \ref{thm:Lojasiewicz_gradient_inequality_generalized_Morse-Bott_function}]
\label{rmk:Barta}  
An example explained to me by Tom\'{a}\v{s} B\'{a}rta indicates the need for condition \eqref{item:Generalized_Morse-Bott_function_lower_order_zero} in Definition \ref{defn:Generalized_Morse-Bott_function} to hold for all $x \in \Crit\sE$ and not just at the point $x_\infty$ in order for the conclusion of Theorem \ref{thm:Lojasiewicz_gradient_inequality_generalized_Morse-Bott_function} to be valid: Choose $d=2$, $N=3$, and $\sE(x,y) = x^3 + x^2y^5$, so $\Crit\sE$ is the $y$-axis, and consider the gradient inequality at points $(-\frac{2}{3}y^5,y)$ in $\KK^2$. 
\end{rmk}

\begin{rmk}[Comparison of Theorem \ref{thm:Lojasiewicz_gradient_inequality_generalized_Morse-Bott_function} and a theorem due to Huang]
\label{rmk:Huang_theorem_2-4-3}  
Huang states a result \cite[Theorem 2.4.3]{Huang_2006} with a conclusion similar to that of Theorem \ref{thm:Lojasiewicz_gradient_inequality_generalized_Morse-Bott_function} (albeit in a Banach-space setting), but his hypotheses are quite different than those of Theorem \ref{thm:Lojasiewicz_gradient_inequality_generalized_Morse-Bott_function} and his result is better viewed as an extension of his \cite[Proposition 2.7.1]{Huang_2006}. On the one hand, our condition \eqref{item:Generalized_Morse-Bott_function_lower_order_zero} in Definition \ref{defn:Generalized_Morse-Bott_function} is replaced in \cite[Theorem 2.4.3]{Huang_2006} by his less restrictive condition that $\sE^{(n)}(x_\infty) = 0$ for $1 \leq n \leq N-1$; on the other hand, our condition \eqref{item:Generalized_Morse-Bott_function_coercive_N_order_derivative} in Definition \ref{defn:Generalized_Morse-Bott_function} is replaced in \cite[Theorem 2.4.3]{Huang_2006} by his condition that $\sE^{(N)}(x_\infty)v^N \neq 0$ for all nonzero $v \in \Ker\sE''(0)$. Our condition \eqref{item:Generalized_Morse-Bott_function_submanifold} that $\Crit\sE$ be a $C^N$ submanifold of $U$ is not assumed by Huang in his \cite[Theorem 2.4.3]{Huang_2006}. 
\end{rmk}

There are other extensions of the concept of a Morse--Bott function, notably that of Kirwan \cite{Kirwan_cohomology_quotients_symplectic_algebraic_geometry}; Holm and Karshon provide a version of her definition of a \emph{Morse--Bott--Kirwan function} in \cite[Definitions 2.1 and 2.3] {Holm_Karshon_2016} and explore its properties and applications to topology. However, it is unclear whether the relatively simple proofs of Theorems \ref{thm:Lojasiewicz_gradient_inequality_Morse-Bott} or \ref{thm:Lojasiewicz_gradient_inequality_generalized_Morse-Bott_function} would extend to include such Morse--Bott--Kirwan functions.

\subsection{{\L}ojasiewicz gradient inequalities for generalized Morse--Bott functions}
\label{subsec:Lojasiewicz_gradient_inequality_generalized_Morse-Bott_function}
The proof of Theorem \ref{thm:Lojasiewicz_gradient_inequality_Morse-Bott} is similar to (and also simpler than) that of Theorem \ref{thm:Lojasiewicz_gradient_inequality_generalized_Morse-Bott_function} and can be obtained in \cite{Feehan_lojasiewicz_inequality_ground_state}, so we shall confine our attention to the following proof.

\begin{proof}[Proof of Theorem \ref{thm:Lojasiewicz_gradient_inequality_generalized_Morse-Bott_function}]
We begin with several reductions that simplify the proof. First, observe that if $\sE_0:U\to\KK$ is defined by $\sE_0(x) := \sE(x+x_\infty)$, then $\sE_0'(0)=0$, so we may assume without loss of generality that $x_\infty=0$ and relabel $\sE_0$ as $\sE$. Second, let $K := T_{x_\infty}\Crit\sE \subset \KK^d$ and observe that by noting the invariance of the conditions in Definition \ref{defn:Generalized_Morse-Bott_function} under $C^N$ diffeomorphisms and applying a $C^N$ diffeomorphism to a neighborhood of the origin in $\KK^d$ and possibly shrinking $U$, we may assume without loss of generality that $U \cap \Crit\sE = U \cap K$, recalling that $\Crit\sE \subset U$ is a submanifold by the hypothesis that $\sE$ is generalized Morse--Bott of order $N$ at $x_\infty$. Third, observe that if $\sE_0:U\to\KK$ is defined by $\sE_0(x) := \sE(x) - \sE(0)$, then $\sE_0(0)=0$, so we may once again relabel $\sE_0$ as $\sE$ and assume without loss of generality that $\sE(0)=0$.

By the second reduction above, it suffices to consider the cases where
\begin{inparaenum}[\itshape i\upshape)]
\item $U\cap \Crit\sE = (\KK^c\oplus 0) \cap U$, for $d \geq 2$ and $1 \leq c \leq d-1$, or
\item $U\cap \Crit\sE = 0 \in \KK^d$, for $d \geq 1$ and $c=0$.
\end{inparaenum}
By shrinking the open subset $U\subset\KK^d$ if necessary, we may assume that $U$ is convex. Applying the Taylor Formula \cite[p. 349]{Lang_analysis} to a $C^M$ function $f:U\to\KK^k$ (for $k\geq 1$) and integer $M\geq 1$ gives
\begin{multline}
\label{eq:Taylor_formula}
f(x) = f(x_0) + f'(x_0)(x-x_0) + \cdots + \frac{f^{(M-1)}(x_0)}{(M-1)!}(x-x_0)^{M-1}
\\
+ \frac{1}{(M-1)!}\int_{0}^{1} (1-t)^{M-1}f^M(x_0+t(x-x_0))(x-x_0)^M\,dt,
\quad\text{for all } x, x_0 \in U.
\end{multline}
For
\begin{inparaenum}[\itshape i\upshape)]
\item $d\geq 2$ and $0 \leq c \leq d-2$, consider
  \[
    v \in S^{d-1-c} = \{x \in \KK^d: c = 0 \text{ or } x_i = 0 \text{ for } 1 \leq i \leq c \text{ and } x_{c+1}^2+\cdots+x_d^2=1\},
  \]
  and for
\item $d\geq 1$ and $c=d-1$, consider $v=1$.
\end{inparaenum}
If $\sE$ is constant in an open neighborhood of $0 \in \KK^d$, then \eqref{eq:Lojasiewicz_gradient_inequality_N-Morse-Bott_function} obviously holds, so we may assume without loss of generality that $\sE$ is non-constant in an open neighborhood of the origin.

By viewing $\sE^{(N)}(0)\in\Hom_\KK(\otimes^{N} \KK^d,\KK)=\otimes^{N} \KK^{d*}$ and recalling that $\sE$ is generalized Morse--Bott of order $N$, there is a positive constant $\zeta$ such that
\begin{equation}
\label{eq:Lower_upper_bounds_order_N_derivative_E}
|\sE^{(N)}(0)v^{N}| \geq \zeta, \quad\text{for all } v \in S^{d-1-c}.
\end{equation}
By viewing $\sE^{(N)}(0)\in\Hom_\KK(\otimes^{N-1} \KK^d,\KK^{d*})$, we note that
\begin{equation}
\label{eq:Dual_space_norm_lower_bound}
\|\sE^{(N)}(0)v^{N-1}\|_{\KK^{d*}} = \max_{w\in S^{d-1-c}}|\sE^{(N)}(0)v^{N-1}w| \geq |\sE^{(N)}(0)v^{N}|, \quad\text{for all } v \in S^{d-1-c}.
\end{equation}
The lower bounds in \eqref{eq:Lower_upper_bounds_order_N_derivative_E} ensure that
\begin{equation}
\label{eq:Lower_upper_bounds_order_N-1_derivative_differetial_E}
\|\sE^{(N)}(0)v^{N-1}\|_{\KK^{d*}} \geq \zeta, \quad\text{for all } v \in S^{d-1-c}.
\end{equation}
Choose small enough positive constants $R$ and $L$ so that the closure of the cylinder $C(R,L) := \{\kappa+rv \in \KK^d: \kappa \in K \text{ with } \|\kappa\|_{\KK^d} < L \text{ and } r \in [0,R) \text{ and } v \in S^{d-1-c}\}$ is contained in $U$. Because $\sE^{(n)}(\kappa) = 0$ for $n=1,\ldots,N-1$ and all $\kappa \in U\cap K$, the Taylor Formula \eqref{eq:Taylor_formula} applied to $f(x)=\sE(x)$ with $k=1$ and $M=N$ and $x_0 = \kappa \in B_L \cap K$ and $x=\kappa+rv \in C(R,L)$ gives
\[
  \sE(\kappa+rv) = \frac{r^N}{(N-1)!}\int_{0}^{1} (1-t)^{N-1}\sE^N(\kappa+trv)v^N\,dt,
\]
or equivalently,
\begin{equation}
\label{eq:Taylor_series}
\sE(\kappa+rv) = \frac{r^N}{N!}\sE^N(0)v^N + \frac{r^N}{(N-1)!}\int_{0}^{1} (1-t)^{N-1}\left(\sE^N(\kappa+trv)-\sE^N(0)\right)v^N\,dt.
\end{equation}
Since $\sE$ is $C^N$, we may choose $R,L \in (0,1]$ small enough that
\begin{equation}
\label{eq:Remainder_Taylor_series}
\sup_{\begin{subarray}{c}s\in[0,R), \\ \|\kappa\|_{\KK^d} < L\end{subarray}}
\left|\left(\sE^N(\kappa+sv)-\sE^N(0)\right)v^N\right|
\leq
|\sE^{(N)}(0)v^{N}|, \quad\text{for all } v \in S^{d-1-c}.
\end{equation}
Therefore, by \eqref{eq:Taylor_series} and \eqref{eq:Remainder_Taylor_series} we obtain
\begin{equation}
\label{eq:Function_raw_upper_bound}
\frac{2r^{N}}{N!}|\sE^{(N)}(0)v^{N}| \geq |\sE(\kappa+rv)|,
\quad\text{for all } v \in S^{d-1-c} \text{ and } r \in [0,R) \text{ and } \kappa \in B_L\cap K.
\end{equation}
As $\sE^{(n)}(\kappa)v^n = 0$ for $n=1,\ldots,N-1$ and all $\kappa \in U\cap K$ and $v \in S^{d-1-c}$, the Taylor Formula \eqref{eq:Taylor_formula} applied to $f(x) = \sE'(x)$ with $k=d$ and $M=N-1$ and $x=\kappa+rv$ and $x_0=\kappa$ yields
\[
\sE'(\kappa+rv) = \frac{r^{N-1}}{(N-2)!}\int_{0}^{1} (1-t)^{N-2}\sE^N(\kappa+trv)v^{N-1}\,dt,
\]
or equivalently,
\begin{multline}
\label{eq:Gradient_is_order_Nminus1_remainder_Taylor_formula}
\sE'(\kappa+rv) = \frac{\sE^{(N)}(0)v^{N-1}}{(N-1)!}r^{N-1}
\\
+ \frac{r^{N-1}}{(N-2)!}\int_{0}^{1} (1-t)^{N-2}\left(\sE^N(\kappa+trv) - \sE^{(N)}(0)\right)v^{N-1}\,dt,
\\
\quad\text{for all } v \in S^{d-1-c} \text{ and } r \in [0,R) \text{ and } \kappa \in B_L\cap K.
\end{multline}
Since $\sE$ is $C^N$, we may choose $R,L \in (0,1]$ small enough that
\begin{equation}
\label{eq:Size_sigma_N}
\sup_{\begin{subarray}{c}s\in[0,R), \\ \|\kappa\|_{\KK^d} < L\end{subarray}}\left\|\left(\sE^N(\kappa+sv) - \sE^{(N)}(0)\right)v^{N-1}\right\|_{\KK^{d*}}
\leq \frac{1}{2}\|\sE^{(N)}(0)v^{N-1}\|_{\KK^{d*}},
\quad\text{for all } v \in S^{d-1-c}.
\end{equation}
Therefore, by \eqref{eq:Gradient_is_order_Nminus1_remainder_Taylor_formula} and \eqref{eq:Size_sigma_N},
\begin{multline}
\label{eq:Derivative_raw_lower_bound}
\|\sE'(\kappa+rv)\|_{\KK^{d*}} \geq \frac{r^{N-1}}{2(N-1)!} \|\sE^{(N)}(0)v^{N-1}\|_{\KK^{d*}},
\\
\text{for all } v \in S^{d-1-c} \text{ and } r \in [0,R) \text{ and } \kappa \in B_L\cap K.
\end{multline}
We compute that, for all $v \in S^{d-1-c}$ and $r \in [0,R)$ and $\kappa \in B_L\cap K$,
\begin{align*}
\|\sE'(\kappa+rv)\|_{\KK^{d*}} &\geq \frac{r^{N-1}}{2(N-1)!} \|\sE^{(N)}(0)v^{N-1}\|_{\KK^{d*}}
\quad\text{(by \eqref{eq:Derivative_raw_lower_bound})}
\\
&= \frac{N}{4}\frac{2}{N!} \|\sE^{(N)}(0)v^{N-1}\|_{\KK^{d*}} \left(\frac{2}{N!}\|\sE^{(N)}(0)v^{N-1}\|_{\KK^{d*}}\right)^{-(N-1)/N}
\\
&\quad \times
\left(\frac{2r^{N}}{N!}\|\sE^{(N)}(0)v^{N-1}\|_{\KK^{d*}}\right)^{(N-1)/N}
\\
&= \frac{N}{4}\left(\frac{2}{N!}\|\sE^{(N)}(0)v^{N-1}\|_{\KK^{d*}}\right)^{1/N}
\left(\frac{2r^{N}}{N!}\|\sE^{(N)}(0)v^{N-1}\|_{\KK^{d*}}\right)^{(N-1)/N}
\\
&\geq \frac{N}{4}\left(\frac{2}{N!}\|\sE^{(N)}(0)v^{N-1}\|_{\KK^{d*}}\right)^{1/N}
\left(\frac{2r^{N}}{N!}|\sE^{(N)}(0)v^{N}|\right)^{(N-1)/N}
\quad\text{(by \eqref{eq:Dual_space_norm_lower_bound})}
\\
&\geq \frac{N}{4}\left(\frac{2}{N!}\|\sE^{(N)}(0)v^{N-1}\|_{\KK^{d*}}\right)^{1/N} |\sE(\kappa+rv)|^{(N-1)/N} \quad\text{(by \eqref{eq:Function_raw_upper_bound})}.
\end{align*}
This yields \eqref{eq:Lojasiewicz_gradient_inequality_N-Morse-Bott_function} with $\theta=(N-1)/N \in [1/2,1)$, for all $v \in S^{d-1-c}$ and $\kappa \in B_L\cap K$, and
\begin{equation}
\label{eq:Definition_c}
C := \frac{N}{4}\inf_{v \in S^{d-1-c}} \left(\frac{2}{N!}\|\sE^{(N)}(0)v^{N-1}\|_{\KK^{d*}}\right)^{1/N}
\geq
\frac{N}{4}\left(\frac{2\zeta}{N!}\right)^{1/N},
\end{equation}
where we apply the lower bound \eqref{eq:Lower_upper_bounds_order_N-1_derivative_differetial_E} to obtain the inequality in \eqref{eq:Definition_c}. By the reductions described earlier, this completes the proof of Theorem \ref{thm:Lojasiewicz_gradient_inequality_generalized_Morse-Bott_function}.
\end{proof}

\section{{\L}ojasiewicz gradient inequality for $C^1$ functions with simple normal crossings}
\label{sec:Lojasiewicz_gradient_inequality_critical_normal_crossing_divisor}
In this section, we prove Theorem \ref{mainthm:Lojasiewicz_gradient_inequality_normal_crossing_divisor} using a simple, coordinate-based alternative to an argument due to Bierstone and Milman of their more general \cite[Theorem 2.7]{Bierstone_Milman_1997}.

\begin{proof}[Proof of Theorem \ref{mainthm:Lojasiewicz_gradient_inequality_normal_crossing_divisor}]
By hypothesis, the function $\sE:U\to\KK$ has simple normal crossings in the sense of Definition \ref{defn:C1_function_normal_crossings} and $\sE(0)=0$. Therefore,
\begin{equation}
\label{eq:Energy_function_normal_crossings}
\sE(x) = \sF(x)\prod_{i=1}^{c}x_i^{n_i}, \quad\text{for all } x \in U,
\end{equation}
for integers $c \geq 1$ with $c \leq d$ and $n_i \geq 1$ and a $C^1$ function $\sF:U\to\KK$ with $\sF(x)\neq 0$ for all $x \in U$.\footnote{By making a further coordinate change, one could assume that $\sF=1$ without loss of generality but we shall omit that step.} Hence, if $\{e_i\}_{i=1}^{d}$ and $\{e_i^*\}_{i=1}^{d}$ denote the standard basis and dual basis, respectively, for $\KK^d$ and $\KK^{d*}$, then the differential of $\sE$ is given by
\begin{align*}
\sE'(x) &= \sum_{j=1}^{d}\sE_{x_j}(x)e_j^*
\\
&= \sum_{j=1}^{c}( x_j^{n_j}\sF_{x_j}(x) + n_j x_j^{n_j-1}\sF(x)) \prod_{\begin{subarray}{c}i=1\\i\neq j\end{subarray}}^{c}x_i^{n_i}\,e_j^* + \prod_{i=1}^{c}x_i^{n_i}\sum_{j=c+1}^{d}\sF_{x_j}(x)e_j^*,
\end{align*}
that is,
\begin{equation}
\label{eq:Gradient_function_normal_crossings}
\sE'(x) = \prod_{i=1}^{c}x_i^{n_i} \sum_{j=1}^{c}(x_j\sF_{x_j}(x) + n_j\sF(x))x_j^{-1} e_j^* + \prod_{i=1}^{c}x_i^{n_i}\sum_{j=c+1}^{d}\sF_{x_j}(x)e_j^*,
\quad\text{for all } x \in U,
\end{equation}
where the sum over $j=c+1,\ldots,d$ is omitted if $c=d$. Observe that
\begin{equation}
\label{eq:Norm_squared_gradient_function_normal_crossings}
\|\sE'(x)\|_{\KK^{d*}}^2
\geq \prod_{i=1}^{c}x_i^{2n_i} \sum_{j=1}^{c}(x_j\sF_{x_j}(x) + n_j\sF(x))^2 x_j^{-2}, \quad\text{for all } x \in U.
\end{equation}
Because $\sF(0)\neq 0$ and $\sF$ is $C^1$, there is a constant $\sigma \in (0,1]$ such that $B_\sigma \Subset U$ and
\[
|x_j\sF_{x_j}(x)| \leq \frac{n_j}{2}|\sF(x)|, \quad\text{for all } x \in B_\sigma \text{ and } j =1,\ldots,c,
\]
and thus
\[
|x_j\sF_{x_j}(x) + n_j\sF(x)| \geq \frac{n_j}{2}|\sF(x)|, \quad\text{for all } x \in B_\sigma \text{ and } j =1,\ldots,c.
\]
Hence \eqref{eq:Norm_squared_gradient_function_normal_crossings}, noting that $n_j\geq 1$ for $j =1,\ldots,c$, yields the lower bound
\begin{equation}
\label{eq:Gradient_squared_lower_bound}
\|\sE'(x)\|_{\KK^{d*}}^2 \geq  \frac{\sF(x)^2}{4} \prod_{i=1}^{c}x_i^{2n_i}\sum_{j=1}^{c}x_j^{-2},
\quad\text{for all } x \in B_\sigma.
\end{equation}
On the other hand, \eqref{eq:Energy_function_normal_crossings} gives
\begin{equation}
\label{eq:Function_squared}
\sE(x)^2 = \sF(x)^2 \prod_{i=1}^{c}x_i^{2n_i}, \quad\text{for all } x \in U.
\end{equation}
Define
\begin{equation}
\label{eq:Max_min_unit}
m := \inf_{x\in B_\sigma}|\sF(x)| > 0 \quad\text{and}\quad M := \sup_{x\in B_\sigma}|\sF(x)| < \infty.
\end{equation}
Because $\sE'(0)=0$, we must have $c\geq 2$ or $c=1$ and $n_1\geq 2$ by examining the expression \eqref{eq:Gradient_function_normal_crossings} for $\sE'(x)$ when $x=0$. If $c=1$, then $n_1\geq 2$ and inequalities \eqref{eq:Gradient_squared_lower_bound}, \eqref{eq:Function_squared}, and \eqref{eq:Max_min_unit} give
\[
\|\sE'(x)\|_{\KK^{d*}} \geq \frac{1}{2m}|x_1|^{n_1-1} \quad\text{and}\quad |\sE(x)| \leq M|x_1|^{n_1}, \quad\text{for all } x \in B_\sigma.
\]
Combining these inequalities yields
\[
\|\sE'(x)\|_{\KK^{d*}} \geq \frac{m}{2M^{(n_1-1)/n_1}}|\sE(x)|^{(n_1-1)/n_1}, \quad\text{for all } x \in B_\sigma,
\]
and hence we obtain \eqref{eq:Lojasiewicz_gradient_inequality_normal_crossing_divisor} with $\theta = 1-1/n_1$ and $C_0 = m/(2M^\theta)$ if $c=1$.

For the remainder of the proof, we assume $c\geq 2$ and recall the \emph{Generalized Young Inequality},
\begin{equation}
\label{eq:Generalized_Young_inequality}
\left(\prod_{j=1}^{c}a_j\right)^r \leq r\sum_{j=1}^{c} \frac{a_j^{p_j}}{p_j},
\end{equation}
for constants $a_j>0$ and $p_j>0$ and $r>0$ such that $\sum_{j=1}^{c} 1/p_j = 1/r$ (see Remark \ref{rmk:Generalized_Young_inequality}). For
\[
N := \sum_{j=1}^{c}n_j,
\]
we observe that the inequality,
\begin{equation}
\label{eq:Haraux_2005_3-4}
\prod_{j=1}^{c} x^{-2n_j/N} \leq \frac{1}{N}\sum_{j=1}^{c}n_jx_j^{-2},
\quad\text{for } x_j\neq 0 \text{ with } j = 1, \ldots, c,
\end{equation}
follows from \eqref{eq:Generalized_Young_inequality} by substituting $r=1$ and $a_j=x_j^{-2n_j/N}$ (with $x_j\neq 0$) and $p_j = N/n_j$ for $j=1,\ldots,c$ in \eqref{eq:Generalized_Young_inequality}. Setting
\[
n := \max_{1\leq j\leq c} n_j  \quad\text{and}\quad \theta := 1-1/N \in [1/2,1)
\]
and applying \eqref{eq:Haraux_2005_3-4} yields
\[
\prod_{i=1}^{c}x_i^{2n_i}\sum_{j=1}^{c}x_j^{-2}
\geq
\frac{N}{n}\prod_{i=1}^{c}x_i^{2n_i(1-1/N)},
\]
that is,
\begin{equation}
\label{eq:Monomial_gradient_inequality}
\prod_{i=1}^{c}x_i^{2n_i}\sum_{j=1}^{c}x_j^{-2}
\geq
\frac{N}{n}\left(\prod_{i=1}^{c}x_i^{2n_i}\right)^\theta,
\quad\text{for all } x \in \KK^c.
\end{equation}
We now combine inequalities \eqref{eq:Gradient_squared_lower_bound}, \eqref{eq:Function_squared}, \eqref{eq:Max_min_unit}, and \eqref{eq:Monomial_gradient_inequality} to give
\[
\|\sE'(x)\|_{\KK^{d*}}^2
\geq
\frac{m^2N}{4n}\left(\prod_{i=1}^{c}x_i^{2n_i}\right)^\theta
\quad\text{and}\quad
\sE(x)^{2\theta}
\leq
M^{2\theta}\left(\prod_{i=1}^{c}x_i^{2n_i}\right)^\theta,
\quad\text{for all } x \in B_\sigma.
\]
Taking square roots and combining the preceding two inequalities yields \eqref{eq:Lojasiewicz_gradient_inequality_normal_crossing_divisor} with constant $C_0 = m\sqrt{N/n}/(2M^\theta)$ if $c\geq 2$. This completes the proof of Theorem \ref{mainthm:Lojasiewicz_gradient_inequality_normal_crossing_divisor}.
\end{proof}

\begin{rmk}[Generalized Young Inequality]
\label{rmk:Generalized_Young_inequality}
The inequality \eqref{eq:Generalized_Young_inequality} may be deduced from Hardy, Littlewood, and P{\'o}lya \cite[Inequality (2.5.2)]{HardyLittlewoodPolya},
\begin{equation}
\label{eq:Hardy_Littlewood_Polya_2-5-2}
\prod_{i=1}^{c}b_i^{q_i} \leq \sum_{i=1}^{c}q_i b_i,
\end{equation}
where $b_i > 0$ and $c\geq 1$ and $q_i>0$ and $\sum_{i=1}^{c}q_i=1$. Indeed, set $a_i = b_i^{q_i/r}$, so $b_i = a_i^{r/q_i}$, and $p_i=r/q_i$ to give
\[
\prod_{i=1}^{c}a_i^r \leq \sum_{i=1}^{c}q_i a_i^{p_i}.
\]
But $q_i = r/p_i$ and thus
\[
\left(\prod_{i=1}^{c}a_i\right)^r \leq r\sum_{i=1}^{c} \frac{1}{p_i} a_i^{p_i},
\]
which is \eqref{eq:Generalized_Young_inequality}; see also \cite[Section 8.3]{HardyLittlewoodPolya}. The inequality \eqref{eq:Generalized_Young_inequality} is proved directly by Haraux as \cite[Lemma 3.2]{Haraux_2005} by using concavity of the logarithm function on $(0,\infty)$.
\end{rmk}

\section{Resolution of singularities and application to the {\L}ojasiewicz gradient inequality}
\label{sec:Resolution_singularities_and_Lojasiewicz_gradient_inequality}
We begin in Sections \ref{subsec:Normal_crossing_divisors} and \ref{subsec:Ideals_smiple_normal_crossing_divisors} by recalling the definitions of divisors and ideals, respectively, with simple normal crossings. In Section \ref{subsec:Resolution_singularities}, we recall a statement of resolution of singularities for analytic varieties and in Section \ref{subsec:Application_Lojasiewicz_gradient_inequality}, we apply that to prove Theorem \ref{mainthm:Lojasiewicz_gradient_inequality} as a corollary of Theorem \ref{mainthm:Lojasiewicz_gradient_inequality_normal_crossing_divisor}. Unless stated otherwise, `analytic' may refer to real or complex analytic in this section.

\subsection{Divisors with simple normal crossings}
\label{subsec:Normal_crossing_divisors}
For basic methods of and notions in algebraic geometry --- including blowing up, divisors, and morphisms --- we refer to Griffiths and Harris \cite{GriffithsHarris}, Hartshorne \cite{Hartshorne_algebraic_geometry}, and Shafarevich \cite{Shafarevich_v1, Shafarevich_v2}. For terminology regarding real analytic varieties, we refer to Guaraldo, Macr\`\i, and Tancredi \cite{Guaraldo_Macri_Tancredi_topics_real_analytic_spaces}; see also Griffiths and Harris \cite{GriffithsHarris} and Grauert and Remmert \cite{Grauert_Remmert_analytic_spaces} for complex analytic varieties.

Following Griffiths and Harris \cite[pp. 12--14, pp. 20--22, and pp. 129--130]{GriffithsHarris} (who consider \emph{complex} analytic subvarieties of smooth \emph{complex} manifolds), let $M$ be a (real or complex) analytic (not necessarily compact) manifold of dimension $d \geq 1$ and $V \subset M$ be an \emph{analytic subvariety}, that is, for each point $p \in V$, there are an open neighborhood $U \subset M$ of $p$ and a finite collection, $\{f_1,\ldots,f_k\}$ (where $k$ may depend on $p$), of analytic functions on $U$ such that $V\cap U = f_1^{-1}(0)\cap \cdots \cap f_k^{-1}(0)$. One calls $p$ a \emph{smooth point of $V$} if $V\cap U$ is cut out transversely by $\{f_1,\ldots,f_k\}$, that is, if the $k\times d$ matrix $(\partial f_i/\partial x_j)(p)$ has rank $k$, in which case (possibly after shrinking $U$), we have that $V\cap U$ is an analytic (smooth) submanifold of codimension $k$ in $U$. An analytic subvariety $V \subset M$ is called \emph{irreducible} if $V$ cannot be written as the union of two analytic subvarieties, $V_1, V_2 \subset M$, with $V_i \neq V$ for $i=1,2$.

One calls $V \subset M$ an \emph{analytic subvariety of dimension $d-1$} if $V$ is a \emph{analytic hypersurface}, that is, for any point $p \in V$, then $U\cap V = f^{-1}(0)$, for some open neighborhood, $U \subset M$ of $p$, and some analytic function, $f$, on $U$ \cite[p. 20]{GriffithsHarris}. We then recall the

\begin{defn}[Divisor on an analytic manifold]
(See \cite[p. 130]{GriffithsHarris}.)
A \emph{divisor} $D$ on an analytic manifold $M$ is a locally finite, formal linear combination,
\[
D = \sum_{i} a_iV_i,
\]
of irreducible, analytic hypersurfaces of $M$, where $a_i \in \ZZ$.
\end{defn}

We can now state the

\begin{defn}[Simple normal crossing divisor]
\label{defn:Subvariety_having simple_normal_crossings_with_divisor}
(See Koll{\'a}r \cite[Definition 3.24]{Kollar_lectures_resolution_singularities}.)
Let $X$ be a smooth algebraic variety of dimension $d\geq 1$. One says that $E = \sum_i E_i$ is a \emph{simple normal crossing divisor} on $X$ if each $E_i$ is smooth and for each point $p \in X$ one can choose local coordinates $x_1,\ldots,x_d$ in the maximal ideal $\fm_p$ of the local ring, $\sO_p$, of regular functions defined on some open neighborhood $U$ of $p\in X$ such that for each $i$ the following hold:
\begin{enumerate}
\item Either $p \notin E_i$ or $E_i\cap U = \{q \in U: x_{j_i}(q) = 0\}$ in an open neighborhood $U \subset X$ of $p$ for some $j_i$, and

\item $j_i \neq j_{i'}$ if $i \neq i'$.
\end{enumerate}
A subvariety $Z \subset X$ has \emph{simple normal crossings} with $E$ if one can choose $x_1,\ldots,x_d$ as above such that in addition
\begin{enumerate}
\setcounter{enumi}{3}
\item $Z = \{q \in U: x_{j_1}(q) = \cdots = x_{j_s}(q) = 0\}$ for some $j_1,\ldots,j_s$.
\end{enumerate}
In particular, $Z$ is smooth, and some of the $E_i$ are allowed to contain $Z$.
\end{defn}

Koll{\'a}r also gives the following, more elementary definition that serves, in part, to help compare the concepts of simple normal crossing divisor (as used by \cite{Kollar_1999, Wlodarczyk_2008}) and normal crossing divisor (as used by \cite{Bierstone_Milman_1997}), in the context of resolution of singularities.

\begin{defn}[Simple normal crossing divisor]
\label{defn:Simple_normal_crossing_divisor}
(See Koll{\'a}r \cite[Definition 1.44]{Kollar_lectures_resolution_singularities}.) Let $X$ be a smooth algebraic variety of dimension $d\geq 1$ and $E \subset X$ a divisor. One calls $E$ a \emph{simple normal crossing divisor} if every irreducible component of $E$ is smooth and all intersections are transverse. That is, for every point $p \in E$ we can choose local coordinates $x_1,\ldots,x_d$ on an open neighborhood $U \subset X$ of $p$ and $m_i \in \ZZ\cap [0,\infty)$ for $i=1,\ldots,d$ such that $U\cap E = \{q \in U: \prod_{i=1}^d x^{m_i}(q) = 0 \}$.
\end{defn}

\begin{rmk}[Normal crossing divisor]
\label{rmk:Normal_crossing_divisor}
(See Koll{\'a}r \cite[Remark 1.45]{Kollar_lectures_resolution_singularities}.)
Continuing the notation of Definition \ref{defn:Simple_normal_crossing_divisor},
one calls $E$ a \emph{normal crossing divisor} if for every $p \in E$ there are local \emph{analytic} or formal coordinates, $x_1,\ldots,x_d$, and natural numbers $m_1,\ldots,m_d$ such that $U\cap E = \{q \in U: \prod_{i=1}^d x^{m_i}(q) = 0 \}$.
\end{rmk}

Definitions \ref{defn:Subvariety_having simple_normal_crossings_with_divisor} and \ref{defn:Simple_normal_crossing_divisor} extend to the categories of analytic varieties, where $\sO_p$ is then the local ring of analytic functions; see, for example, Koll{\'a}r \cite[Section 3.44]{Kollar_lectures_resolution_singularities}. In the category of analytic varieties, Remark \ref{rmk:Normal_crossing_divisor} implies that the concepts of simple normal crossing divisor and normal crossing divisor coincide.\footnote{I am grateful to Jaros{\l}aw W{\l}odarczyk for clarifying this point.} Definitions of simple normal crossing divisors are also provided by Cutkowsky \cite[Exercise 3.13 (2)]{Cutkosky_resolution_singularities}, Hartshorne \cite[Remark 3.8.1]{Hartshorne_algebraic_geometry} and Lazarsfeld \cite[Definition 4.1.1]{Lazarsfeld_positivity_algebraic_geometry_v1}.

\subsection{Ideals with simple normal crossings}
\label{subsec:Ideals_smiple_normal_crossing_divisors}
For our application to the proof of the gradient inequality, we shall need to more generally consider ideals with simple normal crossings and the corresponding statement of resolution of singularities. We review the concepts that we shall require for this purpose. For the theory of ringed spaces, sheaf theory, analytic spaces, and analytic manifolds we refer to Grauert and Remmert \cite{Grauert_Remmert_coherent_analytic_sheaves}, Griffiths and Harris \cite{GriffithsHarris}, and Narasimhan \cite{Narasimhan_introduction_theory_analytic_spaces} in the complex analytic category and Guaraldo, Macr\`\i, and Tancredi \cite{Guaraldo_Macri_Tancredi_topics_real_analytic_spaces} in the real analytic category; see also Hironaka et al. \cite{Aroca_Hironaka_Vicente_theory_maximal_contact, Aroca_Hironaka_Vicente_desingularization_theorems, Hironaka_infinitely_near_singular_points}. If $X$ is an analytic manifold, then $\sO_X$ is the sheaf of analytic functions on $X$. An ideal $\sI \subset \sO_X$ is \emph{locally finite} if for every point $p \in X$, there are an open neighborhood $U \subset X$ and a finite set of analytic functions $\{f_1,\ldots,f_k\} \subset \sO_U$ such that
\[
\sI = f_1\sO_U + \cdots + f_k\sO_U,
\]
and $\sI$ is \emph{locally principal} if $k=1$ for each point $p \in X$.

If $p \in X$, then $\sO_p$ is the ring of (germs of) analytic functions defined on some open neighborhood of $p$. The quotient sheaf $\sO_X/\sI$ is a sheaf of rings on $X$ and its \emph{support}
\[
Z: = \supp(\sO_X/\sI)
\]
is the set of all points $p \in X$ where $(\sO_X/\sI)_p \neq 0$, that is, where $\sI_p \neq \sO_p$. In an open neighborhood $U$ of $p$ one has
\[
Z\cap U = f_1^{-1}(0)\cap \cdots \cap f_k^{-1}(0),
\]
so locally $Z$ is the zero set of finitely many analytic functions.

In order to state the version of resolution of singularities that we shall need, we recall some definitions from Cutkosky \cite[pp. 40--41]{Cutkosky_resolution_singularities} and
Koll{\'a}r \cite[Note on Terminology 3.16]{Kollar_lectures_resolution_singularities}, given here in the real or complex analytic category, rather than the algebraic category, for consistency with our application. Suppose that $X$ is a non-singular variety and $\sI \subset \sO_X$ is an ideal sheaf; a \emph{principalization of the ideal} $\sI$ is a proper birational morphism $\pi : \widetilde{X} \to X$ such that $\widetilde{X}$ is non-singular and
\[
\pi^*\sI \subset \sO_{\widetilde{X}}
\]
is a locally principal ideal. If $X$ is a non-singular variety of dimension $d$ and $\sI \subset \sO_X$ is a locally principal ideal, then one says that $\sI$ has \emph{simple normal crossings} (or is \emph{monomial}) \emph{at a point} $p \in X$ if there exist local coordinates $\{x_1,\ldots,x_d\} \subset \sO_p$ such that
\[
\sI_p = x_1^{m_1}\cdots x_n^{m_d}\sO_p,
\]
for some $m_i \in \ZZ\cap [0,\infty)$ with $i=1,\ldots,d$. One says that $\sI$ is \emph{locally monomial} if it is monomial at every point $p \in X$ or, equivalently, if it is the ideal sheaf of a \emph{simple normal crossing divisor} in the sense of Definition \ref{defn:Subvariety_having simple_normal_crossings_with_divisor}.

Suppose that $D$ is an effective divisor on a non-singular variety $X$ of dimension $n$, so $D = m_1E_1 + \cdots + m_dE_d$, where $E_i$ are irreducible, codimension-one subvarieties of $X$, and $m_i \in \ZZ\cap [0,\infty)$ with $i=1,\ldots,d$. One says that $D$ has \emph{simple normal crossings} if
\[
\sI_D = \sI_{E_1}^{m_1}\cdots\sI_{E_d}^{m_d}
\]
has \emph{simple normal crossings}.

\subsection{Resolution of singularities}
\label{subsec:Resolution_singularities}
We recall from Cutkosky \cite[pp. 40--41]{Cutkosky_resolution_singularities} that a \emph{resolution of singularities} of an algebraic or analytic variety $X$ is a proper birational morphism $\pi : \widetilde{X} \to X$ such that $\widetilde{X}$ is non-singular. Hironaka \cite{Hironaka_1964-I-II} proved that any algebraic variety over any field of characteristic zero admits a resolution of singularities and, moreover, that both complex and real analytic varieties admit resolutions of singularities as well \cite{Aroca_Hironaka_Vicente_theory_maximal_contact, Aroca_Hironaka_Vicente_desingularization_theorems, Hironaka_infinitely_near_singular_points}. Bierstone and Milman \cite{Bierstone_Milman_1997} (see \cite{Bierstone_Milman_1999} for their expository introduction to \cite{Bierstone_Milman_1997}) have developed a proof of resolution of singularities that applies to real and complex analytic varieties and to algebraic varieties over any field of characteristic zero and which significantly shortens and simplifies Hironaka's proof. Additional references for resolution of singularities include Cutkowsky \cite{Cutkosky_resolution_singularities}, Faber and Hauser \cite{Faber_Hauser_2010}, Hauser \cite{Hauser_2003}, Hironaka \cite{Hironaka_1963}, Koll\'ar \cite{Kollar_lectures_resolution_singularities}, Villamayor \cite{Villamayor_1989, Villamayor_1992, Encinas_Villamayor_1998}, and W{\l}odarczyk \cite{Wlodarczyk_2005, Wlodarczyk_2008}. Proofs of special cases of resolution of singularities for real and complex analytic varieties were previously provided by Bierstone and Milman \cite{BierstoneMilman, Bierstone_Milman_1989}. The most useful version of resolution of singularities for our application is

\begin{thm}[Principalization and monomialization of an ideal sheaf]
\label{thm:Monomialization_ideal_sheaf}
(See Bierstone and Milman \cite[Theorem 1.10]{Bierstone_Milman_1997}, Koll{\'a}r \cite[Theorems 3.21 and 3.26 and p. 135 and Section 3.44]{Kollar_lectures_resolution_singularities} and W{\l}odarczyk \cite[Theorem 2.0.2]{Wlodarczyk_2008} for analytic varieties; compare W{\l}odarczyk \cite[Theorem 1.0.1]{Wlodarczyk_2005} for algebraic varieties.)
If $X$ is a smooth analytic variety and $\sI \subset \sO_X$ is a nonzero ideal sheaf, then there are a smooth analytic variety $\widetilde{X}$ and a birational and projective morphism $\pi : \widetilde{X} \to X$ such that
\begin{enumerate}
  \item $\pi^*\sI \subset \sO_{\widetilde{X}}$ is the ideal sheaf of a simple normal crossing divisor,
  \item $\pi : \widetilde{X} \to X$ is an isomorphism over $X \less \cosupp\sI$, where $\cosupp\sI$ (or $\supp(\sO_X/\sI)$) is the cosupport of $\sI$.
\end{enumerate}
\end{thm}

Versions of Theorem \ref{thm:Monomialization_ideal_sheaf} when $X$ is an algebraic surface over a field of characteristic zero are provided by Cutkosky \cite[p. 29]{Cutkosky_resolution_singularities} and Koll{\'a}r \cite[Theorem 1.74]{Kollar_lectures_resolution_singularities}. Kashiwara and Schapira \cite{Kashiwara_Schapira_sheaves_manifolds} provide the following useful variant of Theorem \ref{thm:Monomialization_ideal_sheaf}.

\begin{prop}[Desingularization for the zero set of a real analytic function and its gradient map]
\label{prop:Kashiwara_Schapira_8-2-4}
(See Kashiwara and Schapira \cite[Proposition 8.2.4]{Kashiwara_Schapira_sheaves_manifolds}.) Let $X$ be a real analytic manifold and $f: X \to \RR$ be a real analytic function that is not identically zero on each connected component of $X$. Set $Z = \{x \in X:f(x) = 0 \text{ and } df(x) = 0\}$. Then there exists a proper morphism of real analytic manifolds $\pi:Y\to X$ that induces a real analytic diffeomorphism $Y\less\pi^{-1}(Z) \cong X\less Z$ such that, in an open neighborhood of each point $y_0 \in \pi^{-1}(Z)$, there exist local coordinates $\{y_1,\ldots,y_d\}$ with $f\circ\pi(y) = \pm y_1^{n_1}\cdots y_d^{n_d}$, for some $n_i \in \ZZ\cap[0,\infty)$ with $i=1,\ldots,d$.
\end{prop}

\subsection{Application to the {\L}ojasiewicz gradient inequality}
\label{subsec:Application_Lojasiewicz_gradient_inequality}
We can now conclude the proof of one of our main theorems.

\begin{proof}[Proof of Theorem \ref{mainthm:Lojasiewicz_gradient_inequality}]
As in the proof of Theorem \ref{thm:Lojasiewicz_gradient_inequality_generalized_Morse-Bott_function}, we may assume without loss of generality that $x_\infty= 0$ and $\sE(0)=0\in\KK$. Define $\sI := \sE\sO_U$ to be the ideal in $\sO_U$ generated by $\sE$, with support of $\sO_U/\sI$ given by $Z = \sE^{-1}(0)$. Let $\pi:\widetilde{U}\to U$ be a resolution of singularities provided by Theorem \ref{thm:Monomialization_ideal_sheaf}, so
\[
\pi^*\sI = \tilde{\sE}\sO_{\widetilde{U}}
\]
is the ideal sheaf of a simple normal crossing divisor, where $\tilde{\sE} := \sE\circ\pi$ and
\[
\pi:\widetilde{U}\less E \cong U\less Z
\]
is an analytic diffeomorphism, with
\[
E := \pi^{-1}(Z) = \{\tilde{x}\in \widetilde{U}: \tilde{\sE}(\tilde{x}) = 0\} \subset \widetilde{U}
\]
denoting the exceptional divisor (with ideal $\pi^*\sI$).

By assumption, $0 \in Z$ and we may further assume without loss of generality that $0 \in \pi^{-1}(0) \subset E$ and $\widetilde{U} \subset \KK^d$ is an open neighborhood of the origin, possibly after shrinking $U$ and hence $\widetilde{U}$. By Theorem \ref{thm:Monomialization_ideal_sheaf}, the function $\tilde{\sE}$ is the product of a monomial in the coordinate functions $x_1,\ldots,x_d$ and an analytic function $\sF$ that is non-zero at the origin. In particular, $\tilde{\sE}$ has simple normal crossings in the sense of Definition \ref{defn:C1_function_normal_crossings}, possibly after further shrinking $U$ and hence $\widetilde{U}$, so $\sF(\tilde{x})\neq 0$ for all $\tilde{x} \in \widetilde{U}$. We can thus apply Theorem \ref{mainthm:Lojasiewicz_gradient_inequality_normal_crossing_divisor} to $\tilde{\sE} = \sE\circ\pi$ and obtain
\[
\|(\sE\circ\pi)'(\tilde{x})\|_{\KK^{d*}}
\geq C|(\sE\circ\pi)(\tilde{x})|^\theta, \quad\text{for all } \tilde{x} \in B_\delta,
\]
for constants $C \in (0,\infty)$ and $\theta \in [1/2,1)$ and $\delta \in (0,1]$. Now $(\sE\circ\pi)(\tilde{x}) = \sE(x)$ for $x = \pi(\tilde{x}) \in U$ and therefore the preceding gradient inequality yields
\begin{equation}
\label{eq:Lojasiewicz_gradient_inequality_composition_with_morphism}
\|(\sE\circ\pi)'(\tilde{x})\|_{\KK^{d*}}
\geq C|\sE(x)|^\theta, \quad\text{for all } \tilde x \in B_\delta \text{ and } x = \pi(\tilde x) \in \pi(B_\delta).
\end{equation}
The Chain Rule gives
\begin{align*}
\|(\sE\circ\pi)'(\tilde{x})\|_{\KK^{d*}}
&\leq \|\sE'(\pi(\tilde{x}))\|_{\KK^{d*}} \|\pi'(\tilde{x})\|_{\End(\KK^d)}
\\
&\leq M\|\sE'(\pi(\tilde{x}))\|_{\KK^{d*}}
\quad\text{for all } \tilde{x} \in \widetilde{U},
\end{align*}
where $M := \sup_{\tilde{x} \in B_\delta} \|\pi'(\tilde{x})\|_{\End(\KK^d)}$. Because $\pi(\tilde{x})=x\in U$, the preceding inequality simplifies:
\begin{equation}
\label{eq:Gradient_composition_with_morphism_lower_bound}
\|(\sE\circ\pi)'(\tilde{x})\|_{\KK^{d*}}
\leq
M\|\sE'(x)\|_{\KK^{d*}}, \quad\text{for all } \tilde{x} \in \widetilde{U} \text{ and } x = \pi(\tilde x) \in U.
\end{equation}
The map $\pi$ is open and so $\pi(B_\delta)$ is an open neighborhood of the origin in $\KK^d$ and thus contains a ball $B_\sigma$ for small enough $\sigma\in (0,1]$. By combining the inequalities \eqref{eq:Lojasiewicz_gradient_inequality_composition_with_morphism} and \eqref{eq:Gradient_composition_with_morphism_lower_bound}, we obtain
\[
\|\sE'(x)\|_{\KK^{d*}} \geq (C/M)|\sE(x)|^\theta, \quad\text{for all } x \in B_\sigma,
\]
which is \eqref{eq:Lojasiewicz_gradient_inequality}, as desired.
\end{proof}

We can also complete the proof of one of the main corollaries.

\begin{proof}[Proof of Corollary \ref{maincor:Characterization_optimal_exponent_Morse-Bott_condition}]
From the proof of Theorem \ref{mainthm:Lojasiewicz_gradient_inequality}, the analytic function $\pi^*\sE:\tilde U \to \KK$ has simple normal crossings near the origin in the sense of Definition \ref{defn:C1_function_normal_crossings} and so (after possibly shrinking $U$)
\[
  \pi^*\sE(\tilde x) = \sF(\tilde x)\,\tilde x_1^{n_1}\cdots \tilde x_d^{n_d}, \quad\text{for all } \tilde x \in \tilde U,
\]
where $\sF:\tilde U\to\KK$ is an analytic function such that $\sF(\tilde x)\neq 0$ for all $\tilde x \in \tilde U$ and the $n_i$ are non-negative integers for $i=1,\ldots,d$. Theorem \ref{mainthm:Lojasiewicz_gradient_inequality_normal_crossing_divisor} therefore implies that $\pi^*\sE$ has {\L}ojasiewicz exponent $\theta = 1-1/N$, where $N = \sum_{i=1}^d n_i$ is the total degree of the monomial. In particular, if $\theta=1/2$ then $N=2$ and (after possibly relabeling the coordinates)
\[
  \pi^*\sE(\tilde x) = \sF(\tilde x)\,\tilde x_1^2 \quad\text{or}\quad \pi^*\sE(\tilde x) = \sF(\tilde x)\,\tilde x_1\tilde x_2, \quad\text{for all } \tilde x \in \tilde U.
\]  
Hence, $\pi^*\sE$ is Morse--Bott at the origin in the sense of Definition \ref{defn:Morse-Bott_function}, with
\[
  \Crit\pi^*\sE = \{\tilde x \in \tilde U: \tilde x_1 = 0\} \quad\text{or}\quad \Crit\pi^*\sE = \{\tilde x \in \tilde U: \tilde x_1 = 0 \text{ and } \tilde x_2 = 0\}.
\]
From the proof of Theorem \ref{mainthm:Lojasiewicz_gradient_inequality}, the map $\pi$ is an analytic diffeomorphism from $\tilde U \less (\pi^*\sE)^{-1}(0)$ onto $U \less \sE^{-1}(0)$, where
\[
  \Crit\pi^*\sE  \subset (\pi^*\sE)^{-1}(0) = \{\tilde x \in \tilde U: \tilde x_1 = 0\}
\]
or
\[
  \Crit\pi^*\sE  \subset (\pi^*\sE)^{-1}(0) = \{\tilde x \in \tilde U: \tilde x_1 = 0 \text{ or } \tilde x_2 = 0\}.
\]
In particular, $\pi$ is an analytic diffeomorphism on the complement of a coordinate hyperplane or the union of two coordinate hyperplanes, as claimed.  
\end{proof}

\section{{\L}ojasiewicz distance inequalities}
\label{sec:Lojasiewicz_distance_inequalities}
It remains to prove the distance inequalities (Corollaries \ref{maincor:Lojasiewicz_distance_inequality} and \ref{maincor:Lojasiewicz_gradient-distance_inequality}). For this purpose, the proof of
\cite[Theorem 2.8]{Bierstone_Milman_1997} (see also \cite{Lojasiewicz_1984}) applies but we shall include additional details for completeness. We assume a {\L}ojasiewicz exponent\footnote{We exclude the trivial case $\theta=1$ and $\sE'(0)\neq 0$.} $\theta \in [1/2,1)$, denoted by $\mu=1-\theta\in(0,1/2]$ in \cite{Bierstone_Milman_1997}.

The following result on the convergence of gradient flow is a refinement of a result due to {\L}ojasiewicz (see \cite[Theorem 5]{Lojasiewicz_1963} and \cite[Theorem 1]{Lojasiewicz_1984}, where it is assumed in addition that $\sF$ is analytic).

\begin{thm}[Existence and convergence of solutions to the gradient flow equation]
\label{thm:Gradient_flow}  
Let $d \geq 1$ be an integer, $U \subset \RR^d$ be an open subset, and $\sF:U\to\RR$ be a $C^1$ function such that $\sF(0)=0$ and $\sF'(0)=0$ and $\sF \geq 0$ on $U$ and $\sF$ obeys the {\L}ojasiewicz gradient inequality \eqref{eq:Lojasiewicz_gradient_inequality} with constants $C \in (0, \infty)$ and $\sigma \in (0,1]$ and $\theta \in [1/2,1)$:
\[
\|\sF'(x)\|_{\RR^{d*}} \geq C|\sF(x)|^\theta, \quad\text{for all } x \in B_\sigma.
\]
Then there are a constant $\delta \in (0,\sigma/4]$ and, for each $x \in B_\delta$, a solution, $\bx$ in $C([0,\infty);\RR^d))\cap C^1((0,\infty);\RR^d)$, to 
\begin{equation}
\label{eq:Gradient_flow}
\frac{d\bx}{dt} = -\sF'(\bx(t)) \quad \text{(in $\RR^d$) with } \bx(0)=x,
\end{equation}
such that $\bx(t) \in B_{\sigma/2}$ for all $t\in[0,\infty)$ and $\bx(t) \to \bx_\infty$ in $\RR^d$ as $t\to\infty$, where $\bx_\infty \in B_\sigma\cap\Crit\sE$.
\end{thm}

\begin{proof}
When $\sF$ is analytic (and thus $\sF$ obeys \eqref{eq:Lojasiewicz_gradient_inequality} by Theorem \ref{mainthm:Lojasiewicz_gradient_inequality}), the conclusions were established by {\L}ojasiewicz \cite[Theorem 1]{Lojasiewicz_1984}: Examination of his proof reveals that it is enough to assume that $\sF$ obeys \eqref{eq:Lojasiewicz_gradient_inequality}. The conclusions may also be obtained by specializing \cite[Theorem 4]{Feehan_yang_mills_gradient_flow_v4} to the case of Euclidean space $\RR^d$ (from the Banach and Hilbert space setting considered there) and noting that its hypotheses are fulfilled when $\sF$ is $C^1$ because \eqref{eq:Lojasiewicz_gradient_inequality} holds by hypothesis here, by appealing to the Peano Existence Theorem (see Hartman \cite[Theorem 2.2.1]{Hartman_2002}) for its hypothesis on short-time existence of solutions to \eqref{eq:Gradient_flow}, and by appealing to the integral version \cite[Equation (1.1.2)]{Hartman_2002} of the gradient flow equation \eqref{eq:Gradient_flow},
\[
\by(t) = \by(0) - \int_0^t \sF'(\by(s))\,ds,
\]  
for its hypothesis on estimates for $\|\by(t)-\by(0)\|_{\RR^d}$ for small $t$.
\end{proof}  

We now begin the proof of one of our corollaries.

\begin{proof}[Proof of Corollary \ref{maincor:Lojasiewicz_distance_inequality}]
Consider Item \eqref{item:Distance_critical_set}. Let $\delta \in (0,\sigma/4]$ denote the constant for $\sE$ provided by Theorem \ref{eq:Gradient_flow}. Consider a point $x \in B_\delta$ such that $\sE(x)\neq 0$ and thus $\sE'(x)\neq 0$ by \eqref{eq:Lojasiewicz_gradient_inequality}. Let $T_0 \in (0,\infty]$ be the smallest time such that $\sE'(\bx(T_0))=0$ (and thus $\bx(T_0) \in B_\sigma\cap\Crit\sE$), where $\bx \in C([0,\infty);\RR^d)\cap C^1((0,\infty);\RR^d)$ is the solution to \eqref{eq:Gradient_flow} provided by Theorem \ref{thm:Gradient_flow}, and define the $C^1$ arc-length parameterization function by
\[
  s(t) := \int_0^t \|\dot\bx(t)\|_{\RR^d}\,dt, \quad\text{for all } t \in [0,T_0),
\]
so that $ds/dt = \|\dot\bx(t)\|_{\RR^d} = \|\sE'(\bx(t))\|_{\RR^d}$ by \eqref{eq:Gradient_flow}, denoting $\dot\bx(t) = d\bx/dt$ for convenience. (We use the isometric isomorphism $\RR^d \ni \xi \mapsto (\cdot,\xi)_{\RR^d} \in \RR^{d*}$ to view $\sE'(x)$ as an element of $\RR^d$ or $\RR^{d*}$ according to the context.) Set $S_0 := s(T_0) \in (0,\infty]$ and write $t = t(s)$ for $s \in [0,S_0)$. Define $\by(s) := \bx(t(s))$ and observe that
\[
  \frac{d\by}{ds} = \frac{d\bx}{dt}\frac{dt}{ds} = \frac{d\bx}{dt}\left(\frac{ds}{dt}\right)^{-1} = \frac{d\bx}{dt}\frac{1}{\|\sE'(\bx(t))\|_{\RR^d}} = -\frac{\sE'(\bx(t))}{\|\sE'(\bx(t))\|_{\RR^d}}, \quad\text{for all } t \in (0,T_0),
\]
where we again apply \eqref{eq:Gradient_flow} to obtain the final equality. Hence, $\by \in C([0,S_0);\RR^d)\cap C^1((0,S_0);\RR^d)$ is a solution to the ordinary differential equation,
\begin{equation}
  \label{eq:Gradient_flow_arclength}
  \frac{d\by}{ds} = -\frac{\sE'(\by(s))}{\|\sE'(\by(s))\|_{\RR^d}} \quad \text{(in $\RR^d$) with } \by(0)=x.
\end{equation}
Write $Q(s):=\sE(\by(s))$ and observe that
\begin{align*}
  Q'(s) &= \sE'(\by(s))\by'(s)
  \\
  &= (\by'(s),\sE'(\by(s)))_{\RR^d} \quad\text{(inner product)}
  \\
  &= -\frac{\left(\sE'(\by(s)), \sE'(\by(s))\right)_{\RR^d}}{\|\sE'(\by(s))\|_{\RR^d}}, \quad\text{for all } s\in [0,S_0) \quad\text{(by \eqref{eq:Gradient_flow_arclength}).}
\end{align*}
In particular, we obtain
\begin{equation}
\label{eq:dQds_negative}
Q'(s) = - \|\sE'(\by(s))\|_{\RR^d} < 0, \quad\text{for all } s\in [0,S_0).  
\end{equation}
Now $Q(0) = \sE(x)>0$ (since $\sE\geq 0$ on $U$ by hypothesis and $\sE(x)\neq 0$ by assumption) and $Q(s)\leq Q(0)$ for all $s \in [0,S_0)$ by \eqref{eq:dQds_negative}. But then we have
\begin{align*}
\frac{\sE(x)^{1-\theta}}{1-\theta} &\geq \frac{Q(0)^{1-\theta} - Q(s)^{1-\theta}}{1-\theta}
\\
&= -\frac{1}{1-\theta} \int_{0}^{s} \frac{d}{du}Q(u)^{1-\theta}\,du
\\
&= - \int_{0}^{s} Q(u)^{-\theta}Q'(u)\,du
\\
&= \int_{0}^{s} \sE(\by(u))^{-\theta}\|\sE'(\by(u))\|_{\RR^d}\,du
\\
&\geq \int_{0}^{s} C\,du = Cs \quad\text{for all } 0 \leq s < S_0 \quad\text{(by \eqref{eq:Lojasiewicz_gradient_inequality})}.
\end{align*}
In applying the {\L}ojasiewicz gradient inequality \eqref{eq:Lojasiewicz_gradient_inequality} to obtain the last line above, we relied on the fact that $\by(s) = \bx(t) \in B_{\sigma/2}$ by \eqref{eq:Gradient_flow} for all $t\in [0,T_0)$ or, equivalently, $s \in [0,S_0)$. Therefore,
\begin{equation}
\label{eq:Energy_power_lower_bound}
\frac{\sE(x)^{1-\theta}}{1-\theta} \geq CS_0. 
\end{equation}
It follows that $S_0<\infty$ and thus as $s\uparrow S_0$, the solution $\by(s)$ converges (in $\RR^d$) to a point $\by(S_0) = \bx(T_0) \in \Crit\sE$ in a finite time $S_0$. Moreover, by \eqref{eq:Gradient_flow} we also have $\bx(T_0) \in \bar{B}_{\sigma/2} \subset B_\sigma$. Since $\|\by'(s)\|_{\RR^d} = 1$, then $\by(s)$ is parameterized by arc length and
\begin{align*}
  S_0 &= \Length_{\RR^d}\{\by(s): s\in [0,S_0]\} = \int_0^{S_0}\|\dot\by(s)\|_{\RR^d}\,ds
  \\
      &\geq \|\by(S_0)-\by(0)\|_{\RR^d}
  \\
  &= \|\bx(T_0)-x\|_{\RR^d}
  \\
      &\geq \inf_{z\in B_\sigma\cap\Crit\sE}\|z-x\|_{\RR^d} \quad\text{(since $\bx(T_0) \in B_\sigma\cap\Crit\sE$)}
  \\
      &= \dist_{\RR^d}(x, B_\sigma\cap\Crit\sE).
\end{align*}
From \eqref{eq:Energy_power_lower_bound}, we thus obtain
\[
\sE(x)^{1-\theta} \geq (1-\theta)C\,\dist_{\RR^d}(x, B_\sigma\cap\Crit\sE),
\]
and this is \eqref{eq:Lojasiewicz_distance_inequality_critical_set}, with exponent $\alpha = 1/(1-\theta) \in [2,\infty)$ and positive constant $C_1 = ((1-\theta)C)^{1/(1-\theta)}$. 

We now assume the additional hypothesis that $B_\sigma\cap\Crit\sE \subset B_\sigma\cap\Zero\sE$. Hence,
\begin{align*}
  \dist_{\RR^d}(x, B_\sigma\cap\Crit\sE) &= \inf_{z\in B_\sigma\cap\Crit\sE}\|z-x\|_{\RR^d}
  \\
                                         &\geq \inf_{z\in B_\sigma\cap\Zero\sE}\|z-x\|_{\RR^d}
  \\
  &= \dist_{\RR^d}(x, B_\sigma\cap\Zero\sE).
\end{align*}
Therefore, \eqref{eq:Lojasiewicz_distance_inequality_zero_set} follows from \eqref{eq:Lojasiewicz_distance_inequality_critical_set}. This proves Item \eqref{item:Distance_critical_set}.

Consider Item \eqref{item:Distance_zero_noncritical_set}. We can apply \eqref{eq:Lojasiewicz_distance_inequality_zero_set} to $\sF=\sE^2$ with constants $C_1\in (0,\infty)$ and $\alpha = 1/(1-\theta) \in [2,\infty)$ and $\sigma \in (0,1]$ and $\delta\in(0,\sigma/4]$ determined by $\sF$ to give
\[
\sF(x) \geq C_1\,\dist_{\RR^d}(x, B_\sigma\cap\Zero\sF)^\alpha, \quad\text{for all } x \in B_\delta.
\]
Clearly, $\Zero\sE = \Zero\sF$ and therefore,
\[
\sE(x)^2 \geq C_1\,\dist_{\RR^d}(x, B_\sigma\cap\Zero\sE)^\alpha, \quad\text{for all } x \in B_\delta.
\]
But this is \eqref{eq:Lojasiewicz_distance_inequality_zero_noncritical_set}, as desired, with exponent $\beta = \alpha/2 \in [1,\infty)$ and positive constant $C_2 = \sqrt{C_1}$. This completes the proof of Item \eqref{item:Distance_zero_noncritical_set} and hence Corollary \ref{maincor:Lojasiewicz_distance_inequality}.
\end{proof}  

Next we give the proof of another corollary.

\begin{proof}[Proof of Corollary \ref{maincor:Lojasiewicz_gradient-distance_inequality}]
We may assume without loss of generality that $x_\infty=0$ and $\sE(0)=0$. Note that $B_\sigma\cap\Crit\sE \subset B_\sigma\cap \Zero\sE$, for small enough $\sigma \in (0,1]$, by Theorem \ref{mainthm:Lojasiewicz_gradient_inequality}. We combine the {\L}ojasiewicz gradient and distance inequalities, \eqref{eq:Lojasiewicz_gradient_inequality} and \eqref{eq:Lojasiewicz_distance_inequality_critical_set}, to give for $\delta \in (0,\sigma/4]$ and $\alpha = 1/(1-\theta) \in [2,\infty)$,
\begin{align*}
\|\sE'(x)\|_{\RR^{d*}} &\geq C_0|\sE(x)|^\theta
\\
&\geq C_0\left(C_1\dist_{\RR^d}(x, \Crit\sE)^\alpha\right)^\theta,
\quad\text{for all } x\in B_\delta.
\end{align*}
Since $\theta \in [1/2, 1)$, this yields \eqref{eq:Lojasiewicz_gradient-distance_inequality} with $\mu = \alpha\theta = \theta/(1-\theta) \in [1,\infty)$ and $C_2 = C_0C_1^\theta$.
\end{proof}

Finally, we have the proof of the last corollary.

\begin{proof}[Proof of Corollary \ref{maincor:Lojasiewicz_gradient-distance_inequality_analytic}]
We may assume without loss of generality that $x_\infty=0$ and $\sE(0)=0$. Choose $\sF(x) := \|\sE'(x)\|_{\RR^{d*}}^2$ for all $x \in U$ and observe that $\sF:U\to\RR$ is analytic and $\sF(0)=0$, so Item \eqref{item:Distance_zero_noncritical_set} of Corollary \ref{maincor:Lojasiewicz_distance_inequality} applies to $\sF$ by Remark \ref{rmk:Lojasiewicz_distance_inequality_analytic}. Applying \eqref{eq:Lojasiewicz_distance_inequality_zero_noncritical_set} with $\sF$ in place of $\sE$ gives
\[
|\sF(x)| \geq C_2\,\dist_{\RR^d}(x, B_{\sigma_1}\cap\Zero\sF)^{\alpha_1}, \quad\text{for all } x \in B_{\delta_1},
\]
for some $C_2 \in (0,\infty)$ and $\alpha_1 \in [1,\infty)$ and $\sigma_1 \in (0,1]$ and $\delta_1 \in (0,\sigma_1/4]$. Since $\Zero\sF = \Crit\sE$, this gives
\[
\|\sE'(x)\|_{\RR^{d*}}^2 \geq C_2\,\dist_{\RR^d}(x, B_{\sigma_1}\cap\Crit\sE)^{\alpha_1}, \quad\text{for all } x \in B_{\delta_1},  
\]
and taking square roots yields \eqref{eq:Lojasiewicz_gradient-distance_inequality_analytic} with $\gamma = \alpha_1/2 \in [1/2,\infty)$ and $C_3 = \sqrt{C_2}$.
\end{proof}

\appendix

\section{Resolution of singularities for the cusp curve and bounds for its {\L}ojasiewicz exponent}
\label{sec:Resolution_singularities_cusp_curve}
Let $\KK=\RR$ or $\CC$ and recall that the cusp curve, defined as the set of solutions $(x,y)\in\KK^2$ to
\begin{equation}
\label{eq:cusp_curve}
f(x,y) := x^2 - y^3 = 0
\end{equation}
is an elementary example used in many texts on algebraic geometry to illustrate applications of resolution of singularities. For example, see Hauser \cite[Figure 10, p. 333]{Hauser_2003} for a discussion and illustrations for this example and Smith \cite[Section 5]{Smith_jyvaskyla_summer_school} or Smith, Kahanp{\"a}{\"a}, Kek{\"a}l{\"a}inen, and Traves \cite[Chapter 7]{Smith_Kahanpaa_Kekalainen_Traves_invitation_algebraic_geometry}. Our purpose in this Appendix\footnote{I am grateful to Peter Kronheimer for suggesting this example.} is to illustrate the use of resolution of singularities (via repeated blow ups) for $f$ on a neighborhood of the origin $0 \in \KK^2$ to achieve a simple normal crossing function $\Pi^*f$, as predicted by Theorem \ref{thm:Monomialization_ideal_sheaf}. Our exposition of resolution of singularities for this example closely follows that of \cite[Section 5]{Smith_jyvaskyla_summer_school}. 

Let $X := \KK^2$, and $\PP^1$ be the one-dimensional projective space of all lines $\ell \subset \KK^2$, and $Z := \{(x,y)\in\KK^2: x^2-y^3=0\} \subset X$ and let
\[
  Y = \{(p,\ell) \in \KK^2\times \PP^1: p \in \ell\}
\]
be the blow-up of $\KK^2$ at the origin (a smooth algebraic variety of dimension two), where $\pi:Y \ni (p,\ell) \mapsto p \in \KK^2$ is the natural projection, $E := \pi^{-1}(0)$ is the exceptional divisor, and $\pi:Y\less E \to \KK^2\less\{0\}$ is an analytic diffeomorphism. One can show that $Y=\{(x,y,[s,t]) \in \KK^2\times\PP^1: xt-ys=0\}$ (for example, see \cite[Lemma 5.1]{Smith_jyvaskyla_summer_school}). Let $U_1,U_2 \subset \PP^1$ denote the coordinate patches given by $U_1 := \{[s,t]: s\neq 0\}$ with local coordinate $z=t/s$ and $U_2 := \{[s,t]: t\neq 0\}$ with local coordinate $w=s/t$. Define $W_1 := \KK^2\times U_1 = \KK^3$ and $W_2 := \KK^2\times U_2 = \KK^3$. In the chart $W_1$ with coordinates $(x,y,z)$, we have $Y\cap W_1 = \{xz-y=0\}$ and the map
\[
\phi_1:  \KK^2 \ni (x,z) \mapsto (x,xz,z) \in Y
\]
identifies the coordinate neighborhood $Y\cap W_1$ with $\KK^2$; in the chart $W_2$ with coordinates $(x,y,w)$, we have $Y\cap W_2 = \{xz-y=0\}$ and the map
\[
\phi_2:  \KK^2 \ni (w,y) \mapsto (wy,y,w) \in Y
\]
identifies the coordinate neighborhood $Y\cap W_2$ with $\KK^2$. On the overlaps, we have $z = y/x$ and $w = z^{-1} = x/y$. A convenient representation of local coordinates for $Y$ is $\{x, y/x\}$ in one chart and $\{x/y,y\}$ in the other.

We shall describe the blow-up map $\pi:Y\to\KK^2$ in these local coordinates. We view $Y$ as the union of two copies of $\KK^2$, one with coordinates $\{x, z\}$ and the other with coordinates $\{w, y\}$, where $z = w^{-1} = y/x$. Then the pullbacks of $\pi$ by the local coordinate charts $\phi_1,\phi_2$ are given by
\[
  \pi_1:\KK^2 \ni (x,z) \mapsto (x,xz) \in \KK^2,
\]
in the first chart, and
\[
  \pi_2:\KK^2 \ni (w,y) \mapsto (wy,y) \in \KK^2,
\]
in the second. In these coordinates, the exceptional divisor is given by $\{x=0\}$ in the first chart and by $\{y=0\}$ in the second. 

We now describe the sequence of three blow ups required to achieve the monomialization $\Pi^*f$ of $f(x,y)=x^2-y^3$:
\begin{enumerate}
\item Use $(u,v) \mapsto (x,y) = (uv,v)$ to get $Z'_1 = \{u^2v^2 - v^3 = 0\}$ with exceptional divisor $\{v=0\}$. Note that in this local chart for $Y$, the pull-back of the blow-up map $\pi_1:\KK^2\to\KK^2$ is not surjective since the line $\{y=0\}$ (aside from $(x,y)=(0,0)$) is not in the image. In the other local coordinate chart for $Y$, the pull-back of the blow-up map $\pi_2:\KK^2\to\KK^2$ is given by $(a,b) \mapsto (x,y) = (a,ab)$; this map is not surjective either since the line $\{x=0\}$ (aside from $(x,y)=(0,0)$) is not in the image. However, the combined blow-up map $\pi:Y\to\KK^2$ is surjective. In the second coordinate chart, we have $Z'_2 = \{a^2-a^3b^3=0\}$ with exceptional divisor $\{a=0\}$.
\item Use $(r,s) \mapsto (u,v) = (r,rs)$ to get $Z_1'' = \{r^4s^2 - r^3s^3 = 0\}$ with transform of the old exceptional divisor $\{s = 0\}$ and exceptional divisor $\{r=0\}$. In the second coordinate chart, $(c,d) \mapsto (a,b) = (cd,d)$, we get $Z_2'' = \{c^2d^2-c^3d^6 = 0\}$.
\item Use $(\alpha,\beta) \mapsto (r,s) = (\alpha,\alpha\beta)$ to get $Z_1''' = \{\alpha^6\beta^2 - \alpha^6\beta^3 = 0\} = \{\alpha^6\beta^2(1 - \beta) = 0\}$, with transform of old exceptional divisor $\{\beta = 0\}$ and exceptional divisor $\{\alpha=0\}$. In the second coordinate chart, $(g,h) \mapsto (c,d) = (gh,h)$, we get $Z_2''' = \{g^2h^4-g^3h^9 = 0\} = \{g^2h^4(1-gh^5) = 0\}$.
\end{enumerate}

Near $(\alpha,\beta)=(0,0)$, we have $Z''' =  \{f_0(\alpha,\beta)\alpha^6\beta^2 = 0\}$ with $f_0(\alpha,\beta)=1-\beta$. The composition of the preceding changes of variables, $\Pi$, defines an analytic map on $\KK^2$ that restricts to a diffeomorphism onto its image, 
\[
  \Pi:B\less \Pi^{-1}(Z) \ni (\alpha,\beta) \mapsto (x,y) = (\alpha^3\beta, \alpha^2\beta) \in U\less Z,
\]
for the open unit ball $B$ centered at the origin and an open neighborhood $U$ of the origin, where $\Pi^*f(\alpha,\beta) = f_0(\alpha,\beta)\alpha^6\beta^2$ and $\Pi^{-1}(Z) = \{\alpha=0 \text{ or } \beta = 0\}$. For $(x,y) \notin Z$, then $(\alpha,\beta) = (x/y, y^3/x^2)$; the line $\{\beta=1\}$ corresponds to $\{x^2-y^3=0\}$. We may remove the factor $f_0$ by further choosing $\delta = \beta\sqrt{1-\beta}$ near $\beta=0$ to give
\begin{equation}
\label{eq:Resolution_alphabeta=(0,0)}
\Pi^*f(\alpha,\delta) = \alpha^6\delta^2,
\end{equation}
a monomial of total degree $N=8$. According to Theorem \ref{mainthm:Lojasiewicz_gradient_inequality_normal_crossing_divisor}, the monomial $\alpha^6\delta^2$ has {\L}ojasiewicz exponent $1-1/N = 7/8$ and so by the last step of the proof of Theorem \ref{mainthm:Lojasiewicz_gradient_inequality} in Section \ref{subsec:Application_Lojasiewicz_gradient_inequality}, the {\L}ojasiewicz exponent $\theta$ of $f$ obeys $1/2 \leq \theta \leq 7/8$.

Near $(\alpha,\beta)=(0,1)$, that is, near $(\alpha,\gamma)=(0,0)$ when $\gamma = 1-\beta$, we have $Z''' = \{f_0(\alpha,\gamma)\alpha^6\gamma = 0\}$, with $f_0(\alpha,\gamma)=(1-\gamma)^2$. The composition of the preceding changes of variables defines an analytic map on $\KK^2$ that restricts to a diffeomorphism onto its image,
\[
  \varpi:B\less \varpi^{-1}(Z) \ni (\alpha,\gamma) \mapsto (x,y) = (\alpha^3(1-\gamma), \alpha^2(1-\gamma)) \in V\less Z,
\]
for an open neighborhood $V$ of the origin, where $\varpi^*f(\alpha,\gamma) = f_0(\alpha,\gamma)\alpha^6\gamma$ and $\varpi^{-1}(Z) = \{\alpha=0 \text{ or } \gamma = 1\}$. We may remove the factor $f_0$ by further choosing $\eta = \gamma(1-\gamma)^2$ near $\gamma=0$ to give 
\begin{equation}
\label{eq:Resolution_alphabeta=(0,1)}
\varpi^*f(\alpha,\eta) = \alpha^6\eta,
\end{equation}
a monomial of total degree $N=7$. The monomial $\alpha^6\eta$ has {\L}ojasiewicz exponent $1-1/N = 6/7$ and so the {\L}ojasiewicz exponent $\theta$ of $f$ obeys $1/2 \leq \theta \leq 6/7$. This completes our example.

While the preceding example illustrates the role of blowing up, the {\L}ojasiewicz exponent of an isolated critical point or zero can often be computed explicitly. For example, by applying Gwo{\'z}dziewicz \cite[Theorem 1.3]{Gwozdziewicz_1999} and modifying its application in \cite[Example, p. 365, and Example, p. 366]{Gwozdziewicz_1999}, one can show that $\theta=2/3$. See also Krasi\'nski, Oleksik, and P{\l}oski \cite[Proposition 2 and p. 3888 for the definition of weighted homogeneous polynomials]{Krasinski_Oleksik_Ploski_2009}. Note also that the Hessian matrix of $f$ at the origin is given by
\[
  \Hess f(0,0) = \begin{pmatrix} 2 & 0 \\ 0 & 0 \end{pmatrix}
\]
and so $f$ is not Morse--Bott at the origin.  

%
%

\bibliography{/Users/pfeehan/Dropbox/LATEX/Bibinputs/master,/Users/pfeehan/Dropbox/LATEX/Bibinputs/mfpde}
\bibliographystyle{amsplain-nodash}

\end{document}